\documentclass[12pt]{amsart}
\usepackage{fullpage, url, stmaryrd}
\usepackage[all]{xy}

\newtheorem{theorem}{Theorem}[section]
\newtheorem{lemma}[theorem]{Lemma}

\newtheorem{cor}[theorem]{Corollary}
\theoremstyle{definition}
\newtheorem{defn}[theorem]{Definition}
\newtheorem{remark}[theorem]{Remark}
\newtheorem{hypothesis}[theorem]{Hypothesis}

\numberwithin{equation}{theorem}

\newcommand{\CC}{\mathbb{C}}
\newcommand{\FF}{\mathbb{F}}
\newcommand{\QQ}{\mathbb{Q}}
\newcommand{\RR}{\mathbb{R}}
\newcommand{\ZZ}{\mathbb{Z}}

\newcommand{\calE}{\mathcal{E}}
\newcommand{\calF}{\mathcal{F}}
\newcommand{\calH}{\mathcal{H}}
\newcommand{\calL}{\mathcal{L}}
\newcommand{\calM}{\mathcal{M}}
\newcommand{\calO}{\mathcal{O}}
\newcommand{\calR}{\mathcal{R}}

\newcommand{\frako}{\mathfrak{o}}

\newcommand{\bv}{\mathbf{v}}

\DeclareMathOperator{\bd}{bd}

\DeclareMathOperator{\et}{et}
\DeclareMathOperator{\FEt}{\mathbf{FEt}}

\DeclareMathOperator{\GL}{GL}
\DeclareMathOperator{\Gr}{Gr}
\DeclareMathOperator{\Hom}{Hom}
\DeclareMathOperator{\Pfd}{\mathbf{Pfd}}
\DeclareMathOperator{\prof}{prof}
\DeclareMathOperator{\rank}{rank}
\DeclareMathOperator{\Spa}{Spa}
\DeclareMathOperator{\Spd}{Spd}
\DeclareMathOperator{\Spec}{Spec}

\begin{document}

\title{Simple connectivity of Fargues--Fontaine curves}
\thanks{The author was supported by NSF (grants DMS-1501214, DMS-1802161),  UC San Diego (Warschawski Professorship),
and IAS (Visiting Professorship 2018--2019).}
\author{Kiran S. Kedlaya}
\date{December 18, 2020}

\begin{abstract}
We show that the Fargues--Fontaine curve associated to an algebraically closed field of characteristic $p$
is geometrically simply connected; that is, its base extension from $\QQ_p$ to any complete algebraically closed overfield admits no nontrivial connected finite \'etale covering. 
We then deduce from this an analogue for perfectoid spaces (and some related objects)
of Drinfeld's lemma on the fundamental group of a product of schemes in characteristic $p$.
\end{abstract}

\maketitle

Let $F$ be an algebraically closed field of characteristic $p$ which is complete with respect to a nontrivial multiplicative norm. The construction of Fargues--Fontaine \cite{fargues-fontaine} associates to $F$ a pair of geometric objects, a scheme over $\QQ_p$ and an adic space over $\QQ_p$ in the sense of Huber, which together play a central role in $p$-adic Hodge theory.
The two spaces are related by a morphism from the adic space to the scheme which has some formal features of an analytification map, such as a form of the GAGA principle for coherent sheaves.

Various aspects of the geometry of the spaces constructed by Fargues--Fontaine justifies their use of the term
\emph{curves} with reference to these spaces. 
At the level of local geometry, the scheme is noetherian and regular of dimension 1, while (as shown by the present author \cite{kedlaya-noetherian})
the adic space is noetherian and admits a neighborhood basis consisting of the adic spectra of principal ideal domains.
At the level of global geometry, if one defines the \emph{degree} of an effective divisor to be its length as a scheme, then every principal divisor has degree 0 and the degree map defines a bijection of the Picard group with $\ZZ$.
One also has an interpretation of $p$-adic representations of the Galois group of $F$ over a subfield in terms of vector bundles over a corresponding quotient of the curve, in analogy with the theorem of Narasimhan--Seshadri \cite{narasimhan-seshadri} on unitary representations of the fundamental group of a compact Riemann surface.

Going further, one finds indications that the correct analogy is not with arbitrary curves, but with curves of genus 0.
Most notably, the celebrated theorem of Grothendieck that every vector bundle on a projective line splits into line bundles \cite{grothendieck-vb} has a close analogue for Fargues--Fontaine curves, on which every vector bundle splits into summands each of which is the pushforward of a line bundle along a finite \'etale morphism. (This is materially a reformulation of a prior result of the present author \cite{kedlaya-annals, kedlaya-revisited}, but with a much more transparent statement and an independent proof.)
Using this result, Fargues--Fontaine showed (see Lemma~\ref{L:splitting cover discrete})
that the curves satisfy a form of geometric simple connectivity: if one performs a base extension from $\QQ_p$ to a completed algebraic closure $\CC_p$, then the profinite fundamental group of the curve becomes trivial. (A similar observation had previously been made by Weinstein \cite{weinstein-gqp}.) 

The purpose of this paper is to extend the result of Fargues--Fontaine by establishing the following result.
\begin{theorem}\label{T:simple connectivity}
The Fargues--Fontaine curve associated to $F$ is \emph{geometrically simply connected}:
for any complete algebraically closed overfield $K$ of $\mathbb{Q}_p$, the base extension of the curve from $\QQ_p$ to $K$ admits no nontrivial connected finite \'etale covering.
\end{theorem}

We then apply this result to obtain a form of \emph{Drinfeld's lemma} in $p$-adic analytic geometry,
expanding on a line of inquiry initiated by Scholze \cite{s-berkeley} and continued by this author in \cite[Lecture~4]{kedlaya-aws}.
The original statement by this name is a result about schemes of characteristic $p$,
and specifically about the effect of formation of products on fundamental groups. For context,
recall that (under suitable ``reasonableness'' hypotheses) the formation of the profinite \'etale fundamental group of a scheme commutes with taking products over an algebraically closed field of characteristic 0 
(see for example \cite[Corollary~4.1.23]{kedlaya-aws}).
This fails in characteristic $p$ without stronger hypothesis (e.g., properness of one of the factors);
Drinfeld's lemma is a replacement statement in which one forms the fiber product directly over $\FF_p$, rather than over an algebraically closed field, but then takes a formal quotient by the action of relative Frobenius on one of the two factors. See Theorem~\ref{T:drinfeld schemes} for the precise statement.

We prove an analogue of Drinfeld's lemma in which the category of schemes over $\FF_p$ (or equivalently for the statement, of perfect schemes over $\FF_p$) is replaced by the category of perfectoid spaces;
see Theorem~\ref{T:drinfeld lemma perfectoid}. We also obtain a formally stronger statement in which the category of
perfectoid spaces is replaced by a certain category of stacks, which can be Scholze's category of \emph{diamonds}
or the even larger category of \emph{small v-sheaves} (both terms used in the sense of \cite{s-berkeley});
see Theorem~\ref{T:drinfeld lemma v-sheaves}. In both cases, Theorem~\ref{T:simple connectivity} amounts to the special case where the two input spaces are both geometric points, and (as for schemes) it is a formal reduction to pass from this case to the general case.

As an aside, we note that Drinfeld's lemma for perfectoid spaces has some notable consequences in $p$-adic Hodge theory, particularly for giving descriptions of the representation theory of products of $p$-adic Galois groups analogous to the description for singleton Galois groups given by the theory of $(\varphi, \Gamma)$-modules.
See \cite{ckz} for a treatment of this topic.

We conclude this introduction with some words on the proof of Theorem~\ref{T:simple connectivity}.
We begin with two approaches that we were unable to make work. One is to somehow reduce to Drinfeld's lemma for schemes; this fails because the nature of absolute products is rather different in the category of perfectoid spaces than in the category of schemes, and in particular quasicompactness is not preserved (see Lemma~\ref{L:Pfd products}).
Another is to adapt the proof in the case $K = \CC_p$; this fails because the classification of vector bundles on a Fargues--Fontaine curve makes crucial use of the discreteness of the valuation on finite extensions of $\QQ_p$. (An exception to this occurs when $F$ is the completed algebraic closure of $\FF_p((t))$, as then one can deduce the claim from Lemma~\ref{L:splitting cover discrete} using a certain symmetry between $F$ and $K$;
see Remark~\ref{R:symmetry applied} and Remark~\ref{R:symmetry}.)

To circumvent these points, we use a two-pronged approach. First,
we give a direct argument using the Artin--Schreier construction
to rule out the existence of nontrivial $\ZZ/p\ZZ$-covers
(Lemma~\ref{L:no abelian cover}); it may be of independent interest that essentially the same argument applies for schemes (see Lemma~\ref{L:drinfeld schemes pro-p}).
We then use this to establish an inductive statement: if simple connectivity holds for a particular coefficient field $K$, then it is also true for any completed algebraic closure of $K(t)$.
(It should also be possible to prove Drinfeld's lemma for schemes by such an inductive approach, but we do not know of a reference where this is done; see Remark~\ref{R:special case}.)
In light of the argument for $\ZZ/p\ZZ$-covers, this reduces to the construction of ``ramification filtrations'' on representations of the \'etale fundamental group using tools from the theory of $p$-adic differential equations \cite{kedlaya-book};
this is in the spirit of previous work of the author 
\cite{kedlaya-swan1, kedlaya-swan2} and Xiao \cite{xiao-swan1, xiao-swan2} on differential Swan conductors.
(It should be possible to replace the use of $p$-adic differential equations with a somewhat more elementary argument using ramification theory of covers of Berkovich curves, as described in \cite{cohen-temkin-trushkin, temkin-curves}; we did not attempt to do this.)

\subsection*{Acknowledgments} Thanks to Peter Scholze for helpful discussions.

\section{Some algebraic preliminaries}

We start with some assorted algebraic preliminaries.

\begin{lemma} \label{L:integral tensor product}
Let $k$ be an algebraically closed field. Let $A$ and $B$ be two integral $k$-algebras. Then
$A \otimes_k B$ is again integral.
\end{lemma}
\begin{proof}
Since $k$ is algebraically closed, $A$ is an \emph{absolutely integral} $k$-algebra
in the sense of Bourbaki \cite[\S V.17, Proposition 8, Corollaire 2]{bourbaki-algebra}.
By \cite[\S V.17, Proposition 2]{bourbaki-algebra}, this forces $A \otimes_k B$ to be integral.
\end{proof}

\begin{defn}
By a \emph{nonarchimedean} field $F$, we will mean a field equipped with a \emph{specified} nontrivial multiplicative norm with respect to which $F$ is complete. (Specifying a norm on $F$, in addition to its topology, amounts to fixing  the norm of some topologically nilpotent unit.)
Let $\frako_F$, $\kappa_F$, $\Gamma_F$ denote the valuation ring, residue field, and value group of $F$ (written additively), respectively.
\end{defn}

\begin{defn} \label{D:spherically complete}
For $k$ a field and $\Gamma$ a totally ordered subgroup of $\RR$ (written additively), the field of 
Hahn--Mal'cev--Neumann generalized power series $k((t^\Gamma))$ is defined to be the set of formal sums
$\sum_{i \in \Gamma} x_i t^i$ with coefficients in $k$ having well-ordered support.
For $x$ such a sum, we write $x_-, x_0, x_+$ for the sums of $x_i t^i$ over indices $i$ with respectively $i<0$, $i=0$, $i>0$. Some key facts about this construction are the following; these include results of 
Kaplansky \cite{kaplansky, kaplansky2}.
\begin{itemize}
\item
For the $t$-adic absolute value, the field $k((t^\Gamma))$ is not only complete but \emph{spherically complete}
(every decreasing sequence of balls has nonempty intersection).
\item
The field $k((t^\Gamma))$ is algebraically closed if and only if $k$ is algebraically closed and $\Gamma$ is divisible. 
\item
Any nonarchimedean field $L$ of equal characteristic which is spherically complete is isomorphic to
$\kappa_L((t^{\Gamma_L}))$.
\item
A nonarchimedean field is spherically complete if and only if it is \emph{maximally complete}, that is,
it admits no nontrivial algebraic extension with the same value group and residue field. 
\end{itemize}
\end{defn}

\begin{defn} \label{D:types of points}
Let $D$ be a one-dimensional affinoid space over an algebraically closed nonarchimedean field $K$, in the sense of Huber's category of adic spaces \cite{huber-book}. We classify points of $D$ into \emph{types 1,2,3,4,5} as in \cite[Example~2.20]{scholze1}; note that this classification is preserved by lifting a point along a finite morphism.
\end{defn}

We recall also the following additional features of the classification. 
\begin{lemma}
With notation as in Definition~\ref{D:types of points}, the following statements hold.
\begin{enumerate}
\item
Each point of type 1 corresponds to a classical rigid-analytic point of $D$.
\item
Each point of type 2 is the generic point on some one-dimensional affinoid space of good reduction. The reduction of this affinoid space is a curve $C$ over the residue field $\kappa$ of $K$; if $C$ is proper over $\kappa$, we say that $D$ is a \emph{strict neighborhood} of the original point. 
The genus of the smooth compactification of $C$ is called the \emph{residual genus} of the original point.
\item
Each point of type 3 is the intersection of a descending sequence of annuli.
\item
Each point of type 4 is the intersection of a descending sequence of discs.
\item
Each point of type 5 is a specialization of a type 2 point; we may thus again speak of the \emph{residual genus} of a type 5 point.
The points of type 5 are the only points whose associated valuations are of rank 2 rather than rank 1. 
\end{enumerate}
\end{lemma}
\begin{proof}
The maximal Hausdorff quotient of $D$ is a one-dimensional Berkovich analytic space, which consists of points of type 1--4 with the specified properties;
see \cite[Example~1.4.1]{berkovich} for the case where $D$ is contained in the projective line and
\cite[\S 3.3]{ducros} for the general case.

To complete the argument from this, it suffices to check that every point of $D$ not in the associated Berkovich space is a specialization of a type 2 point. For this, one may reduce 
to the case where $D$ is a disc or annulus and then argue as in \cite[Example~2.20]{scholze1}.
\end{proof}

\section{Drinfeld's lemma for schemes}

We next recall from \cite[\S 4.2]{kedlaya-aws} the formulation of Drinfeld's lemma for schemes.
This will be used primarily as the basis of an analogy with the setup for perfectoid spaces;
our approach to Drinfeld's lemma for perfectoid spaces does not depend on the full result for schemes,
but only on a special case which can be handled more directly (Lemma~\ref{L:drinfeld schemes pro-p}).
On the other hand, it is possible to go the other way, deriving the result from schemes from the result for perfectoid spaces; see \S\ref{sec:drinfeld diamonds}.

\begin{defn}
Let $X$ be a scheme or adic space, and let $\Gamma$ be a group of automorphisms of $X$.
We say that $X$ is \emph{$\Gamma$-connected} if $X$ is nonempty and its only $\Gamma$-stable closed-open subsets are itself and the empty set. 

Let $\FEt(X/\Gamma)$ be the category of finite \'etale coverings of $X$
equipped with $\Gamma$-actions (i.e., finite \'etale coverings of the stack-theoretic quotient of $X$ by $\Gamma$).
This is a Galois category  in the sense of \cite[Tag~0BMQ]{stacks-project}.

For $X$ $\Gamma$-connected and $\overline{x}$ a geometric point of $X$, let
$\pi_1^{\prof}(X/\Gamma, \overline{x})$ be the automorphism group of the fiber functor $Y \mapsto \left| Y_{\overline{x}} \right|$ on $\FEt(X/\Gamma)$. As usual, the choice of the basepoint $\overline{x}$
is needed to resolve a conjugation ambiguity in the definition; when this ambiguity is not an issue, we may omit the choice of $\overline{x}$ from the notation. 

When $\Gamma$ is the cyclic group generated by a single automorphism $\varphi$, we typically write
$X/\varphi$ in place of $X/\Gamma$.
When $\Gamma$ is the trivial group, we typically write $X$ in place of $X/\Gamma$; this recovers the usual definition of the
profinite (\'etale) fundamental group.
\end{defn}

\begin{remark}
In the setting of adic spaces, a geometric point of $X$ is a morphism $\Spa(K,K^+) \to X$ where $K$ is algebraically closed, but we do not require that $K^+$ equal $K^\circ$; that is, the valuation ring of $K$ may have rank greater than 1. However, the choice of $K^+$ has no effect on the resulting fiber functor.
\end{remark}

\begin{lemma} \label{L:same pi1}
Let $f: Y \to X$ be a morphism of schemes or adic spaces, such that both $X$ and $Y$ are qcqs 
(quasicompact and quasiseparated). 
Suppose that the base change functor $\FEt(X) \to \FEt(Y)$ is an equivalence of categories.
\begin{enumerate}
\item[(a)]
The map $\pi_0(X) \to \pi_0(Y)$ is a homeomorphism.
\item[(b)]
Suppose that one of $X$ or $Y$ is connected. Then so is the other, and for any geometric point $\overline{y}$ of $Y$ the map $\pi_1^{\prof}(Y, \overline{y}) \to \pi_1^{\prof}(X, \overline{y})$ is a homeomorphism.
\end{enumerate}
\end{lemma}
\begin{proof}
For schemes, this is \cite[Tag~0BQA]{stacks-project}; the same proof applies for adic spaces
once we observe that the underlying topological space of a qcqs adic space is a spectral space.
(This is true by construction for adic affinoid spaces, and a qcqs adic space is covered by finitely
many adic affinoid spaces spaces with the pairwise intersections being quasicompact.)
\end{proof}

We now restrict to the case of schemes. For the corresponding discussion for analytic spaces,
see \S\ref{sec:drinfeld diamonds}.
\begin{defn} \label{D:Phi quotient}
Let $X_1,\dots,X_n$ be  schemes over $\FF_p$ and put
$X = X_1 \times_{\FF_p} \cdots \times_{\FF_p} X_n$. 
Write $\varphi_i$ as shorthand for $\varphi_{X_i}$, the automorphism of $X$ induced by the absolute ($p$-power) Frobenius on $X_i$.
We say that $X$ is \emph{$\Phi$-connected} if $X$ is $\langle \varphi_1,\dots, \widehat{\varphi_i}, \dots,\varphi_n \rangle$-connected for some (and hence any) $i \in \{1,\dots,n\}$. 

Define the category
\[
\FEt(X/\Phi) := \FEt(X/ \langle \varphi_1, \dots, \varphi_n \rangle)
\times_{\FEt(X/\varphi_X)} \FEt(X);
\]
for each $i \in \{1,\dots,n\}$, there is a canonical equivalence
\[
\FEt(X/\Phi) \cong \FEt(X/\langle \varphi_1,\dots, \widehat{\varphi_i}, \dots,\varphi_n \rangle).
\]
\end{defn}

\begin{remark}
The condition that $X$ is $\Phi$-connected is equivalent to saying that the algebraic stack $X/\Phi$ is connected;
similarly, $\FEt(X/\Phi)$ is naturally (in $X$) equivalent to the category of finite \'etale coverings of the stack $X/\Phi$. Of course $X/\Phi$ is generally not a scheme; by contrast, the corresponding construction for perfectoid spaces will give an object in the same category (see Lemma~\ref{L:Pfd products}).
\end{remark}

The key nonformal input into the proof of Drinfeld's lemma is the following.
\begin{lemma} \label{L:drinfeld schemes point}
Let $X$ be any scheme over $\FF_p$, let $k$ be an algebraically closed field over $\FF_p$, put
$X_k := X \times_{\FF_p} k$, and let $\varphi_k: X_k \to X_k$ be the morphism induced by the Frobenius on $k$. Then
the functor
\[
\FEt(X) \to \FEt(X_k/\varphi_k)
\]
is an equivalence of categories.
\end{lemma}
\begin{proof}
See \cite[Lemma~4.2.6]{kedlaya-aws}.
\end{proof}

This then leads to the precise statement of Drinfeld's lemma;
note that the additional hypotheses on $X_1,\dots,X_n$ arise from Lemma~\ref{L:same pi1}.
For the mechanism to pass from Lemma~\ref{L:drinfeld schemes point} and this result, see the proof of Theorem~\ref{T:drinfeld lemma perfectoid}.
\begin{theorem}[Drinfeld's lemma for schemes] \label{T:drinfeld schemes}
Let $X_1,\dots,X_n$ be connected qcqs (quasicompact quasiseparated) schemes over $\FF_p$ and put
$X = X_1 \times_{\FF_p} \cdots \times_{\FF_p} X_n$. 
\begin{enumerate}
\item[(a)]
The scheme $X$ is $\Phi$-connected.
\item[(b)]
For any geometric point
$\overline{x}$ of $X$, the map
\[
\pi_1^{\prof}(X/\Phi, \overline{x}) \to \prod_{i=1}^n \pi_1^{\prof}(X_i, \overline{x})
\]
is an isomorphism of topological groups.
\end{enumerate}
\end{theorem}
\begin{proof}
See \cite[Lemma~4.2.11]{kedlaya-aws} for (a) and \cite[Theorem~4.2.12]{kedlaya-aws} for (b).
\end{proof}

\begin{remark} \label{R:drinfeld schemes perfect}
To strengthen the analogy we are going after, we note that the forgetful functor from perfect schemes over $\FF_p$ (i.e., those on which Frobenius is an isomorphism) to arbitrary schemes over $\FF_p$ admits a right adjoint which preserves the Zariski and \'etale topologies, corresponding to the functor on rings over $\FF_p$ given by
$R \mapsto \varinjlim_{\varphi} R$. 
Consequently, Theorem~\ref{T:drinfeld schemes} does not change if we require $X_1,\dots,X_n$ to be perfect (although the proof goes through schemes of finite type over $\FF_p$).
\end{remark}

\begin{remark} \label{R:special case}
One can give a direct proof of Lemma~\ref{L:drinfeld schemes point} in the case where $k$ is an algebraic closure of $\FF_p$. In this case, any given object of $\FEt(X_k/\varphi_k)$ is the base extension of an object of $\FEt(X \times_{\FF_p} \ell)$ for some finite extension $\ell$ of $\FF_p$, and thus in turn may be viewed as an object of $\FEt(X)$. 
One shows easily that pulling back from $X$ to this covering splits the the original object of $\FEt(X_k/\varphi_k)$.

With this in mind, one plausible approach to proving Lemma~\ref{L:drinfeld schemes point} in general would be to show that the statement for a given $k$ implies the same for an algebraic closure of $k(t)$. While we do not know of a reference for a proof along these lines, our approach to Drinfeld's lemma for analytic spaces follows essentially this model, with 
Lemma~\ref{L:splitting cover discrete} playing the role of the base case. (A more precise analogue of the argument would be do this induction in the case where $X$ is a geometric point, then formally promote the statement to general $X$; compare the proof of Lemma~\ref{L:extend by geometric point}.)
\end{remark}

As discussed above, we do not know of a way to apply Theorem~\ref{T:drinfeld schemes} directly to deduce
Drinfeld's lemma for perfectoid spaces. However, our proof of the latter does rely upon a limited case of 
Theorem~\ref{T:drinfeld schemes}, covering the maximal pro-$p$ quotient of the profinite fundamental group.
For this, one can give a rather direct argument, which we will then expand upon in the adic setting
(\S\ref{sec:abelian}); the fact that this case of Drinfeld's lemma is easy to prove, even in the schematic case, may be of independent interest.

\begin{lemma} \label{L:drinfeld schemes pro-p}
Let $k_1, k_2$ be algebraically closed fields of characteristic $p$,
put $X_i = \Spec(k_i)$, and put $X := \Spec(k_1 \otimes_{\FF_p} k_2)$.
Then $X$ is $\Phi$-connected and for any geometric point $\overline{x}$ of $X$, 
the maximal pro-$p$ quotient of $\pi_1^{\prof}(X/\Phi, \overline{x})$
is trivial.
\end{lemma}
\begin{proof}
Recall that $(k_1 \otimes_{\FF_p} k_2)^{\varphi}$ is the set of continuous functions from
$X$ to $\FF_p$ (see for example \cite[Corollary~3.1.4]{kedlaya-liu1}).
Since $k_2$ is flat over $\FF_p$,
\[
(k_1 \otimes_{\FF_p} k_2)^{\varphi,\varphi_1} 
= (k_1^{\varphi_1} \otimes_{\FF_p} k_2)^{\varphi} 
= (\FF_p \otimes_{\FF_p} k_2)^{\varphi_2} = k_2^{\varphi_2} = \FF_p;
\]
this implies that $X$ is $\Phi$-connected.

Since $k_1$ and $k_2$ are algebraically closed, we have Artin-Schreier exact sequences
\[
0 \to \FF_p \to k_i \stackrel{\varphi_1-1}{\to} k_i \to 0 \qquad (i=1,2).
\]
By tensoring the sequence with $i=1$ over $\FF_p$ with $k_2$, we obtain a commutative diagram with exact rows
\[
\xymatrix{
0 \ar[r] & k_2 \ar[r] \ar^{\varphi-1}[d] & k_1 \otimes_{\FF_p} k_2 \ar^{\varphi_1-1}[r] \ar^{\varphi-1}[d] & 
k_1 \otimes_{\FF_p} k_2 \ar[r] \ar^{\varphi-1}[d] & 0 \\
0 \ar[r] & k_2 \ar[r] & k_1 \otimes_{\FF_p} k_2 \ar^{\varphi_1-1}[r] & 
k_1 \otimes_{\FF_p} k_2 \ar[r] & 0 
}
\]
to which we may apply the snake lemma; this yields a surjective morphism
\[
0 = H^1(\varphi, k_2) \to H^1(\varphi, k_1 \otimes_{\FF_p} k_2)^{\varphi_1}
\]
whose target may be identified with $H^1((X/\Phi)_{\et}, \ZZ/p\ZZ)$ by Artin-Schreier again.
This proves that $X/\Phi$ has no nonsplit \'etale $\ZZ/p\ZZ$-cover, from which the claim follows.
\end{proof}

\section{Fargues--Fontaine curves}
\label{sec:FF curves}

We continue with various notations and statements about Fargues--Fontaine curves. We use standard notation for Huber rings and pairs and their adic spectra, as in \cite{weinstein-aws}.
 
\begin{hypothesis} \label{H:Robba rings}
For the remainder of the paper (except as specified),
let $L$ be an algebraically closed nonarchimedean field of characteristic $p$.
Fix a power $q$ of $p$ and an embedding $\FF_q \hookrightarrow \kappa_L$
(which lifts uniquely to an embedding $\FF_q \hookrightarrow L$)
and let $E$ be a local field with residue field $\FF_q$. Let $\varpi$ be a uniformizer of $E$.
Let $F$ be an algebraically closed nonarchimedean field containing $E$.
\end{hypothesis}

\begin{remark}
In \cite[Hypothesis~2.1]{kedlaya-noetherian} it is only asserted that $E$ must be a complete discretely valued field whose residue field \emph{contains} $\FF_q$, but almost every subsequent statement requires this containment to be an equality (as in \cite{fargues-fontaine}).
For instance, in \cite[Definition~2.2]{kedlaya-noetherian}, the expression of a general element of
$W(\frako_L) \otimes_{W(\FF_q)} \frako_E$ as a sum $\sum \varpi^n [\overline{x}_n]$ with $\overline{x}_n \in L$
depends on $E$ having residue field $\FF_q$.
\end{remark}

\begin{defn} \label{D:norms}
For $R$ a perfect $\FF_q$-algebra, define $W(R)_E := W(R) \otimes_{W(\FF_q)} \frako_E$.
For $I \subseteq (0, \infty)$ a closed interval, let $B^I_{L,E}$ denote the Fr\'echet completion of $W(\frako_L)_E[\varpi^{-1}][[\overline{x}]: \overline{x} \in L]$ for the family of multiplicative 
(see Lemma~\ref{L:curve over extension field}) norms
\[
\lambda_t\left( \sum_{n \in \ZZ} \varpi^n [\overline{x}_n]  \right) = \max\{ p^{-n} \left| \overline{x}_n \right|^t \}
\qquad (t \in I);
\]
this ring is a principal ideal domain \cite[Theorem~7.11]{kedlaya-noetherian}
 and a strongly noetherian Huber ring \cite[Theorem~4.10]{kedlaya-noetherian}.
For any given $x \in B^I_{L,E}$, the function $t \mapsto \log \lambda_t(x)$ is convex on $I$ \cite[Lemma~4.4]{kedlaya-noetherian} (see also Lemma~\ref{L:curve over extension field}), so $B^I_{L,E}$ is in fact a Banach ring for the norm $\max\{\lambda_r, \lambda_s\}$ for $I = [r,s]$. 
\end{defn}

\begin{lemma} \label{L:curve over extension field}
Let $\overline{\FF}_q$ be the algebraic closure of $\FF_q$ in $F$. Choose an extension of the chosen embedding $\FF_q \hookrightarrow L$ to an embedding $\overline{\FF}_q \hookrightarrow L$.
\begin{enumerate}
\item[(a)]
For $t>0$, the tensor product norm on $B^{[t,t]}_{L,E} \widehat{\otimes}_{W(\overline{\FF}_q)_E} \frako_F$ is multiplicative.
We again denote this norm by $\lambda_t$.
\item[(b)]
For $x \in B^I_{L,E} \widehat{\otimes}_{W(\overline{\FF}_q)_E} \frako_F$, the function $t \mapsto \log \lambda_t(x)$ on $I$ is continuous and convex.
\end{enumerate}
\end{lemma}
\begin{proof}
For any ring $A$ equipped with a submultiplicative norm $\alpha$, define the associated graded ring $\Gr A$
 by the formula
\[
\Gr A = \bigoplus_{r>0} \Gr^r A, \qquad \Gr^r A = \frac{\{x \in A: \alpha(x) \leq r\}}{\{x \in A: \alpha(x) < r \}}.
\]
The ring $\Gr L$ consists of one graded component for each $r$ in the value group $\left| L^\times \right|$, each of which is a one-dimensional vector space over $\kappa_L$. Since $L$ is a nonarchimedean field, $\Gr L$ is an integral domain. We may then write
\[
\Gr B^{[t,t]}_{L,E} \cong (\Gr L)[\overline{\varpi}]
\]
with $\Gr L$ rescaled by $t$ (that is, place $\Gr^r$ in degree $r^t$ rather than $r$)
and $\overline{\varpi}$ placed in degree $p^{-1}$. This is again an integral domain. Finally, 
by Lemma~\ref{L:integral tensor product}, 
\[
\Gr (B^{[t,t]}_{L,E} \widehat{\otimes}_{W(\overline{\FF}_q)_E} \frako_F)
\cong (\Gr L)[\overline{\varpi}^{\Gamma_L}] \otimes_{\overline{\FF}_q} \kappa_F
\]
is integral, and so the tensor product norm on $B^{[t,t]}_{L,E} \widehat{\otimes}_{W(\overline{\FF}_q)_E} \frako_F$ is multiplicative. This yields (a).

To check (b), we may work locally around a single $t \in I$.
In particular, we may write $x$ as a sum of simple tensors, then ignore any of those that do not
contribute to the image of $x$ in $\Gr (B^{[t,t]}_{L,E} \widehat{\otimes}_{W(\overline{\FF}_q)_E} \frako_F)$.
From the upper and lower degrees of this image, viewed as a Laurent-Puiseux polynomial in $\overline{\varpi}$,
we may read off the slopes of $t \mapsto \log \lambda_t(x)$ on either side of $t$.
\end{proof}

\begin{cor} \label{C:units}
With notation as in Lemma~\ref{L:curve over extension field}, let
$J$ be a (possibly singleton) closed interval contained in the interior of $I$. 
Then $x \in  B^I_{L,E} \widehat{\otimes}_{W(\overline{\FF}_q)_E} \frako_F$
is a unit in $B^J_{L,E} \widehat{\otimes}_{W(\overline{\FF}_q)_E} \frako_F$
if and only if $t \mapsto \log \lambda_t(x)$ is an affine function of $t$ on some neighborhood of $J$ in $I$.
\end{cor}
\begin{proof}
Suppose first that $x$ admits the inverse $y$ in $B^J_{L,E} \widehat{\otimes}_{W(\overline{\FF}_q)_E} \frako_F$. Let $J'$ be some closed interval contained in the interior of $I$ which contains $J$ in its interior.
We can then approximate $y$ by some element $z \in B^{J'}_{L,E} \widehat{\otimes}_{W(\overline{\FF}_q)_E} \frako_F$
in such a way that $\lambda_t(y-z) < \lambda_t(x)^{-1}$ for all $t \in J$.
In particular, $\lambda_t(1-xz) = \lambda_t(x(y-z)) < 1$ for all $t \in J$; by Lemma~\ref{L:curve over extension field}(b), for a suitable choice of $J'$ this remains true for all $t \in J'$. For $t \in J'$, we then have
\[
\lambda_t(x) + \lambda_t(z) = \lambda_t(xz) = 1
\]
but by (b) the functions $t \mapsto \lambda_t(x)$ and $t \mapsto \lambda_t(z)$ are both convex. They must therefore both be affine, proving the claim.

In the other direction, it suffices to check that if $J$ is a singleton interval
and $t \mapsto \log \lambda_t(x)$ is an affine function of $t$ on some neighborhood of $J$ in $I$,
then $x$ is a unit in $B^{J'}_{L,E} \widehat{\otimes}_{W(\overline{\FF}_q)_E} \frako_F$
for some closed interval $J'$ containing $J$ in its interior.
The image of $x$ in $\Gr(B^J_{L,E} \widehat{\otimes}_{W(\overline{\FF}_q)_E} \frako_F)$
must then be an element of $(\Gr L) \otimes_{\overline{\FF}_q} \frako_F$ times some power of $\varpi$.
The element of $(\Gr L) \otimes_{\overline{\FF}_q} \frako_F$ must be placed in a single degree, and hence must be a unit.
For suitable $J'$,
we can then construct $y \in B^I_{L,E} \widehat{\otimes}_{W(\overline{\FF}_q)_E} \frako_F$
for which $\lambda_t(1-xy) < 1$ for $t \in J'$.
Then $xy$ is a unit, as then is $x$.
\end{proof}

\begin{remark}
The ring $B^{I}_{L,E} \widehat{\otimes}_{W(\overline{\FF}_q)_E} \frako_F$ does not share some of the more
refined ring-theoretic properties of $B^I_{L,E}$, essentially due to the value group of $F$ not being discrete. Notably, $B^{I}_{L,E} \widehat{\otimes}_{W(\overline{\FF}_q)_E} \frako_F$ is not noetherian.
\end{remark}

\begin{remark} \label{R:char p curve}
Suppose that $E$ is of characteristic $p$. Then for any Banach $E$-algebra $A$,
the ring $W(\frako_L)_E[\varpi^{-1}][[\overline{x}]: \overline{x} \in L] \otimes_E A$
contains $L \otimes_{\FF_q} A$ as a dense subring.
In the case where $A = F$,
the restriction to $L \otimes_{\FF_q} A$ of the norm $\lambda_t$ from Lemma~\ref{L:curve over extension field} coincides with the tensor product norm for the given norm on $F$ and the $t$-th power of the given norm on $L$.
\end{remark}

\begin{remark} \label{R:interval perfectoid}
For any perfectoid $E$-algebra $A$, $B^I_{L,E} \widehat{\otimes}_E A$ is also perfectoid.
Moreover, for any topologically nilpotent unit $t \in A^\flat$ there is a canonical isomorphism
\[
(B^I_{L,E} \widehat{\otimes}_E A)^\flat \cong
B^I_{L, \FF_q((t))} \widehat{\otimes}_{\FF_q((t))} A^\flat.
\]
\end{remark}

\begin{defn}
Let $Y_{L,E}$ be the inductive limit of the adic spaces $\Spa(B^I_{L,E}, B^{I,\circ}_{L,E})$ as $I$ varies over all
closed intervals in $(0, \infty)$. The $q$-power Frobenius maps $\varphi_L: B^I_{L,E} \to B^{I^{1/q}}_{L,E}$ induces an isomorphism
$\varphi_L^*: Y_{L,E} \to Y_{L,E}$. The group $\varphi_L^{* \ZZ}$ acts properly discontinuously on $Y_{L}$;
define the \emph{adic Fargues--Fontaine curve} $X_{L}$ to be the quotient by this action.
\end{defn}

\begin{remark} \label{R:FF connected}
By Lemma~\ref{L:curve over extension field}, $B^I_{L,E} \widehat{\otimes}_{W(\FF_q)_E} \frako_F$ is not connected; its connected components are each isomorphic to $B^I_{L,E} \widehat{\otimes}_{W(\overline{\FF}_q)_E} \frako_F$
and, as a topological space, form a principal homogenous space for $G_{\FF_q} \cong \widehat{\ZZ}$.
The same description then applies to the connected components of $Y_{L,E} \times_E F$. However, the action of $\varphi_L^*$ on this space is nontrivial: it is via the action of the dense subgroup $\ZZ \subseteq \widehat{\ZZ}$.
As a result, $X_{L,E} \times_E F$ is connected.
\end{remark}

\begin{remark} \label{R:diamond interpretation}
Suppose that $E$ is of characteristic $p$ and that $A$ is a Banach $E$-algebra. 
Let
$R^+$ be the completion of $\frako_L \otimes_{\FF_q} A^\circ$ for the $(\varpi_L, \varpi)$-topology for some
(any) topologically nilpotent unit $\varpi_L$ of $L$, and put $R = R^+[\varpi_L^{-1}, \varpi^{-1}]$.
Then there is a canonical identification
\[
Y_L \times_E \Spa(A,A^\circ) \cong \{v \in \Spa(R,R^+): v(\varpi_L), v(\varpi) < 1\}.
\]
See Lemma~\ref{L:Pfd products} for an expansion of this remark.
\end{remark}

\section{Vector bundles}
\label{sec:vector bundles}

We next gather some statements about vector bundles on Fargues--Fontaine curves,
and deduce the simple connectivity of the Fargues--Fontaine curve associated to $\CC_p$.

\begin{defn} \label{D:twisting bundle1}
For $n \in \ZZ$, let $\calO(n)$ be the line bundle on $X_{L,E}$ corresponding to the $\varphi_L$-equivariant 
line bundle on $Y_{L,E}$ whose underlying bundle is trivial with a generator $\bv$ satisfying
$\varphi_L^* \bv = \varpi^{-n} \bv$.
\end{defn}

\begin{lemma} \label{L:vector bundle to phi-module}
For any adic affinoid space $U$ over $E$, the category of vector bundles on $X_{L,E} \times_{E} U$ is equivalent to the category of $\varphi$-modules over $\calO(Y_{L,E} \times_{E} U)$ (i.e., finite projective modules over the ring $\calO(Y_{L,E} \times_{E} U)$ equipped with isomorphisms with their $\varphi$-pullbacks).
\end{lemma}
\begin{proof}
It is apparent that the category of vector bundles on $X_{L,E} \times_{E} U$ is equivalent to the category of vector bundles on $Y_{L,E} \times_{E} U$ equipped with isomorphisms with their $\varphi$-pullbacks.
Since $Y_{L,E} \times_{E} U$ is a quasi-Stein space in the sense of \cite[\S 2.6]{kedlaya-liu2},
we obtain a fully faithful functor from 
the category of $\varphi$-modules over $\calO(Y_{L,E} \times_{E} U)$ 
to the category of vector bundles on $Y_{L,E} \times_{E} U$ equipped with isomorphisms with their $\varphi$-pullbacks. To check that this is essentially surjective, we must show that given a vector bundle on $Y_{L,E} \times_{E} U$ equipped with isomorphisms with their $\varphi$-pullbacks, the global sections form a finitely generated module over $\calO(Y_{L,E} \times_{E} U)$ (as then \cite[Corollary~2.6.8]{kedlaya-liu2} implies that the module is also projective). For this, 
note that the space $Y_{L,E} \times_{E} U$ admits a locally finite covering by spaces of the form $\Spa(B^{I}_{L,E}, B^{I,\circ}_{L,E}) \times_{E} U$ (e.g., by taking $I = [tq^n, tq^{n+1}]$ for $t$ fixed and $n$ varying over $\ZZ$),
so we may apply \cite[Lemma~2.6.15]{kedlaya-liu2} to conclude.
\end{proof}

\begin{remark}
It is also possible to prove Lemma~\ref{L:vector bundle to phi-module} by showing that $\calO(1)$ is an ample line bundle on $X_{L,E} \times_{E} U$, as in \cite[\S 6]{kedlaya-liu1}. We omit further details here.
\end{remark}

\begin{defn} \label{D:twisting bundle2}
For $n$ a positive integer, let $X_{L,E,n}$ be the quotient of $Y_{L,E}$ by the action of 
$\varphi_L^{* n \ZZ}$. Let $\pi_n: X_{L,E,n} \to X_{L,E}$ be the natural projection; it is a connected $n$-fold \'etale cover which splits upon base extension from $E$ to $F$ (see Remark~\ref{R:FF connected}).

For $d = \frac{r}{s} \in \QQ$ written in lowest terms (so that $r,s \in \ZZ$, $\gcd(r,s) = 1$, and $s>0$),
let $\calO(d)$ be the vector bundle on $X_{L,E}$ defined as follows. Start with 
a trivial line bundle on $Y_{L,E}$ with a generator $\bv$. As in Definition~\ref{D:twisting bundle1},
promote this to a $\varphi_L^s$-equivariant line bundle by specifying that $\varphi_L^{s*} \bv = \varpi^{-r} \bv$;
this descends to a line bundle on $X_{L,E,s}$. Then push forward along $\pi_s$ to obtain $\calO(d)$; for $d \in \ZZ$, this agrees with Definition~\ref{D:twisting bundle1}.
\end{defn}

\begin{lemma} \label{L:twisting sections}
For $d \in \QQ$, the following statements hold.
\begin{enumerate}
\item[(a)]
If $d=0$, then $H^0(X_{L,E}, \calO(d)) = E$.
\item[(b)]
If $d>0$, then $H^0(X_{L,E}, \calO(d)) \neq 0$.
\item[(c)]
If $d>d'$, then $\Hom(\calO(d), \calO(d')) = 0$.
\item[(d)]
For any positive integer $m$, $\calO(d)^{\otimes m}$ is isomorphic to a direct sum of copies of $\calO(dm)$.
\end{enumerate}
\end{lemma}
\begin{proof}
See for instance \cite[\S 8.2]{fargues-fontaine}.
\end{proof}

\begin{theorem} \label{T:vb classification}
Suppose that $L$ is algebraically closed. Then every vector bundle on $X_{L,E}$ splits as a direct sum $\bigoplus_i \calO(d_i)$ for some $d_i \in \QQ$.
\end{theorem}
\begin{proof}
Modulo the interpretation of vector bundles given in Lemma~\ref{L:vector bundle to phi-module}, this result
is originally due to the author when $E$ is of characteristic $0$
(see \cite[Theorem~4.16]{kedlaya-annals} in the case $L = \CC_p^\flat$ and \cite[Theorem~4.5.7]{kedlaya-revisited} in the general case)
and to Hartl--Pink when $E$ is of characteristic $p$ \cite[Theorem~11.1]{hartl-pink}.
The formulation given here is due to Fargues--Fontaine \cite[Th\'eor\`eme 8.2.10]{fargues-fontaine},
who give an independent proof.
See \cite[Theorem~3.6.13]{kedlaya-aws} for further discussion.
\end{proof}

We now use Theorem~\ref{T:vb classification} to establish a base case of simple connectivity,
following Weinstein \cite{weinstein-gqp}, Fargues--Fontaine \cite{fargues-fontaine}, and Scholze \cite{s-berkeley}. See also \cite[Lemma~4.3.10]{kedlaya-aws}.

\begin{lemma} \label{L:splitting cover discrete}
Suppose that $F$ is a completed algebraic closure of $E$.
Then every finite \'etale cover of $X_{L,E} \times_E F$ splits.
\end{lemma}
\begin{proof}
In light of the equivalence
\[
\FEt(X_{L,E} \times_E F) \cong 2\mbox{-}\varinjlim_{E'} \FEt(X_{L,E} \times_E E')
\]
for $E'$ varying over finite extensions of $E$ within $F$
(e.g., apply \cite[Proposition~2.6.8]{kedlaya-liu1} to each term in a finite covering of $X_{L,E}$ by affinoid 
subspaces), we may start with a cover
$f: U \to X_{L,E} \times_E F$ which is the base extension of the cover
$f_0: U_0 \to X_{L,E} \times_E E' = X_{L,E'}$ for some finite extension $E'$ of $F$.
Let $g_0: U_0 \to X_{L,E} \times_E E' \to X_{L,E}$ be the composite projection.
Apply Theorem~\ref{T:vb classification} to split $g_{0 *} \calO_{U_0}$ as a direct sum $\bigoplus_i \calO(d_i)$
for some $d_i \in \QQ$. 

We claim that in fact $d_i = 0$ for all $i$. To see this, suppose by way of contradiction that $d_i > 0$ for some
$i$; since there are only finitely many such $i$, we may assume without loss of generality that $d_i = \max_j \{d_j\}$. By Lemma~\ref{L:twisting sections}, $H^0(X_{L,E}, \calO(d_i))$ is nonzero, and any nonzero element gives rise to a nonzero square-zero
element of $H^0(U_0, \calO_{U_0})$; however, this yields a contradiction because $X_{L,E}$ is reduced, as then must be $U_0$. Hence $d_i \leq 0$ for all $i$; since $g_{0 *} \calO_{U_0}$ is self-dual, we also have $d_i \geq 0$ for all $i$. 

Consequently, $d_i = 0$ for all $i$, and so $H^0(U_0, \calO_{U_0}) = H^0(X_{L,E}, g_{0 *} \calO_{U_0})$
is a finite-dimensional $E$-vector space. This vector space inherits from $U_0$ the structure of an \'etale $E$-algebra; it follows that the original cover $f$ splits.
\end{proof}

\begin{remark} \label{R:symmetry applied}
One may already deduce from Lemma~\ref{L:splitting cover discrete} that for $L = \CC_p^\flat$, $X_{L,E}$ is geometrically simply connected.
To do this, one must use the symmetry between $L$ and $F$ coming from
Remark~\ref{R:diamond interpretation};
see Remark~\ref{R:symmetry} for further discussion.
\end{remark}

\begin{remark} \label{R:Banach on sections}
For any vector bundle $\calE$ on $X_{L,E} \times_E F$, we may pull back to 
$B^I_{L,E} \widehat{\otimes}_E F$ for some interval $I$ to equip $H^0(X_{L,E} \times_E F, \calE)$
with the structure of a Banach module over $F$. This structure does not depend on the choice of $I$,
as follows from the log-convexity of $\lambda_t(x)$ as a function of $t$
(Definition~\ref{D:norms}). More precisely, it is clear that replacing $I$ with $p^n I$ for any $n \in \ZZ$
does not change the Banach module structure; then for any other interval $J \subset (0, \infty)$,
we can find integers $n_1, n_2$ such that $J$ is contained in the convex hull of
$p^{n_1} I \cup p^{n_2} I$.
\end{remark}

\begin{lemma} \label{L:global sections}
The natural map $F \to H^0(X_{L,E} \times_E F, \calO)$ is an isomorphism.
\end{lemma}
\begin{proof}
To check the claim at hand, we may formally reduce to the case where $F$ is the completion of a subfield of countable dimension over $E$. We may then construct a Schauder basis for $F$ over $E$ 
\cite[Proposition~2.7.2/3]{bgr} to reduce to
the assertion that $E \to H^0(X_{L,E}, \calO)$ is an isomorphism, which is
 Lemma~\ref{L:twisting sections}(a).
\end{proof}

\begin{lemma} \label{L:split on overfield}
Let $L'$ be an algebraically closed complete overfield of $L$.
Let $F$ be a complete algebraically closed overfield of $E$ and let $F'$ be a complete algebraically closed overfield of $F$. Let $f: U \to X_{L,E} \times_E F$ be a finite \'etale cover whose base extension to
$X_{L',E} \times_E F'$ splits completely. Then $f$ splits completely.
\end{lemma}
\begin{proof}
Let $f': U' \to X_{L',E} \times_E F'$ be the base extension of $f$.
Using Remark~\ref{R:Banach on sections}, we may equip $H^0(U, \calO) = H^0(X_{L,E} \times_E F, f_* \calO_U)$
with the structure of a Banach algebra over $F$. Using a Schauder basis argument again,
we see that the natural map $H^0(U, \calO) \widehat{\otimes}_F F' \to H^0(U', \calO)$ is an isomorphism
of Banach algebras. Since $f'$ splits completely, by Lemma~\ref{L:global sections}
the ring $H^0(U', \calO)$ is a finite direct sum of copies of $F$.
Consequently,
$H^0(U, \calO)$ is a finite-dimensional reduced $F$-algebra, and hence must itself split completely because $F$ is algebraically closed. This splitting induces the desired splitting of $f$.
\end{proof}

\begin{remark}
Although we have chosen to establish Lemma~\ref{L:splitting cover discrete} separately for historical reasons,
it is not actually logically necessary to do so; the reduction method used to prove
Theorem~\ref{T:simple connectivity2} can further be used to reduce Lemma~\ref{L:splitting cover discrete}
to Remark~\ref{R:special case}. We omit the details here.
\end{remark}

\section{Abelian covers}
\label{sec:abelian}

We next give a direct argument to split $\ZZ/p\ZZ$-covers, picking up the thread from Lemma~\ref{L:drinfeld schemes pro-p}.

\begin{lemma} \label{L:product norm}
Suppose that $E$ is of characteristic $p$ and that $L \cong \kappa_L((u^{\Gamma_L})), F \cong \kappa_F ((t^{\Gamma_F}))$. Assume further that $\Gamma_L, \Gamma_F$ are subgroups of $\RR$ and that the norms on $L,F$ are normalized so that
\[
\left| u^i \right| = p^{-i}, \qquad \left| t^j \right| = p^{-j} \qquad (i \in \Gamma_L, j \in \Gamma_F).
\]
Put $k := \kappa_L \otimes_{\FF_q} \kappa_F$ and let $k^{\Gamma_L \times \Gamma_K}$ be the set
of functions $\Gamma_L \times \Gamma_F \to k$, with elements written as $k$-valued formal sums over $\Gamma_L \times \Gamma_F$.
Consider the map $L \otimes_{\FF_q} F \to k^{\Gamma_L \times \Gamma_K}$ induced by the bilinear map
\[
L \times F \to k^{\Gamma_L \times \Gamma_K}, \quad
\left( \sum_{i \in \Gamma_L} x_i u^i, \sum_{j \in \Gamma_F} y_j t^j \right) \mapsto
\sum_{(i,j) \in \Gamma_L \times \Gamma_F} (x_i \otimes y_j) u^i t^j.
\]
Then for any $r>0$, the restriction to $L \otimes_{\FF_q} F$ of the function $\lambda_r$ on $k^{\Gamma_L \times \Gamma_K}$ given by
\[
\lambda_r \left( \sum_{(i,j) \in \Gamma_L \times \Gamma_F} x_{i,j} u^i t^j \right) = \sup\{p^{-ri-j}: x_{i,j} \neq 0\}
\]
computes the tensor product of the given norm on $F$ and the $r$-th power of the given norm on $L$.
\end{lemma}
\begin{proof}
We start by fixing some terminology in order to articulate the argument. By a \emph{presentation} of an element $z \in L \otimes_{\FF_q} F$, we mean an expression of $z$ as a finite sum $\sum_l x_l\otimes y_l$ of simple tensors in $L \otimes_{\FF_q} F$. For a presentation of the form $\sum_l x_l \otimes y_l$, define the \emph{norm} of the presentation as $\max_l \{\left| x_l \right|^r \left| y_l \right|\}$; by definition, the tensor product seminorm of an element of $L \otimes_{\FF_q} F$ is the infimum of the norms of all presentations of $z$.
For $z$ admitting a presentation as a simple tensor $x \otimes y$, the norm of such a presentation equals $\lambda_r(z)$; it follows formally that $\lambda_r$ is a lower bound for the tensor product norm.
To complete the argument, we will show that every $z \in L \otimes_{\FF_q} F$ admits a presentation with norm $\lambda_r(z)$; as a bonus, that shows that in this situation, the infimum in the definition of the tensor product norm is always achieved (which need not hold in a more general setting).

Let $\sum_l x_l \otimes y_l$ be a presentation of $z$.
Let $L_1$ (resp.\ $F_1$) be the $\FF_q$-vector subspace of $L$ (resp.\ $F$) spanned by the $x_i$ (resp.\ the $y_i$).
Choose a basis $e_1,e_2,\dots$ of $L_1$ which is \emph{normalized} in the following sense:
\begin{itemize}
\item
the valuations of $e_1,e_2,\dots$ form a nondecreasing sequence; and
\item
if $e_{i_1},\dots,e_{i_2}$ have the same valuation, then their images in the graded ring of $L$ are linearly independent over $\FF_q$. That is, the images of $e_i/e_{i_1}$ in $\kappa_L$ for $i=i_1,\dots,i_2$ are linearly independent over $\FF_q$.
\end{itemize}
Similarly, choose a normalized basis $f_1,f_2,\dots$ of $F_1$. Now rewrite the original presentation of $z$ in the form $\sum_{i,j} c_{i,j} e_i \otimes f_j$ with $c_{i,j} \in \FF_q$; this presentation has norm equal to $\lambda_r(z)$.
\end{proof}

\begin{remark} \label{R:sign splitting}
In the notation of Lemma~\ref{L:product norm}, the completion of $L \otimes_{\FF_q} F$ injects into
$k^{\Gamma_L \times \Gamma_K}$. It is a subtle problem to describe the image $R$, and we will not attempt to do so here. One remark we do make is that as a $k$-module, $R$ splits as a direct sum
$\bigoplus_{e_1,e_2 \in \{-,0,+\}} R_{e_1,e_2}$ in which $R_{e_1,e_2}$ is the set of elements of $R$ supported on pairs $(i,j)$ with signs $(e_1, e_2)$: namely, this splitting is obtained by splitting the factors of simple tensors in the form $x = x_+ + x_0 + x_-$ as per Definition~\ref{D:spherically complete}.
\end{remark}

\begin{lemma} \label{L:no abelian cover}
Every \'etale $\ZZ/p\ZZ$-cover of $X_{L,E} \times_E F$ splits.
\end{lemma}
\begin{proof}
By Remark~\ref{R:interval perfectoid}, we may assume that $E$ is of characteristic $p$.
By Lemma~\ref{L:split on overfield}, we may check the claim after enlarging both $L$ and $F$;
we may thus also assume that both fields are spherically complete.
We may then identify $L$ and $F$ with 
$\kappa_L((u^{\Gamma_L}))$ and $\kappa_F ((t^{\Gamma_F}))$, respectively; moreover,
the groups $\Gamma_L,\Gamma_F$ must be divisible and the fields
$\kappa_L, \kappa_F$ must be algebraically closed.

Define the rings $R := \bigcap_I B^I_{L,E} \widehat{\otimes}_E F = \calO(Y_{L,E} \times_E F)$,
$k := \kappa_L \otimes_{\FF_q} \kappa_F$.
Recall that since we are in characteristic $p$, by Remark~\ref{R:char p curve} we may interpret $B^I_{L,E} \widehat{\otimes}_E F$ as a completion of $L \otimes_{\FF_q} F$; as per Lemma~\ref{L:product norm} and
Remark~\ref{R:sign splitting}, we may write any element $x \in R$ 
as a formal sum $\sum_{(i,j) \in \Gamma_L \times \Gamma_F} x_{i,j} u^i t^j$ with
$x_{i,j} \in k$, and use this representation to compute $\lambda_r$ for all $r>0$.
We may further decompose $x$ canonically as $\sum_{e_1,e_2 \in \{-,0,+\}} x_{e_1, e_2}$
as in Remark~\ref{R:sign splitting}.

Recall the general Artin--Schreier construction: on any space of characteristic $p$ we have an exact sequence
of \'etale sheaves
\[
0 \to \underline{\FF_p} \to \calO \stackrel{\varphi-1}{\to} \calO \to 0.
\]
On the quasi-Stein space $Y_{L,E} \times_E F$, this yields isomorphisms
\[
H^i((Y_{L,E} \times_E F)_{\et}, \ZZ/p\ZZ) \cong H^i(\varphi, R) \qquad (i=0,1).
\]
Quotienting by $\varphi_L$ and considering the Hochschild--Serre spectral sequence yields a commutative diagram
\[
\xymatrix{
0 \ar[r] & H^1(\varphi_L,k^{\varphi}) \ar[r] \ar[d] & 
H^1((\Spec(k)/\Phi)_{\et}, \ZZ/p\ZZ) \ar[d]
 \\ 
0 \ar[r] & H^1(\varphi_L,R^{\varphi}) \ar[r] &  H^1((X_{L,E} \times_E F)_{\et}, \ZZ/p\ZZ) \ar[r] & H^1(\varphi, R)^{\varphi_L}.
}
\]
with exact rows. By Lemma~\ref{L:drinfeld schemes pro-p},
the terms in the top row vanish. From formal sums, we see that the left vertical arrow is an isomorphism.
Consequently, to finish we must check that $H^1(\varphi, R)^{\varphi_L} = 0$.

Before proceeding, we record a key observation. Again by the Artin--Schreier construction and the Hochschild--Serre spectral sequence, we have an exact sequence
\[
H^1((\Spec(L \otimes_{\FF_q} F)/\Phi)_{\et}, \ZZ/p\ZZ) \to
H^1(\varphi, L \otimes_{\FF_q} F)^{\varphi_L} \to H^2(\varphi_L, (L \otimes_{\FF_q} F)^{\varphi}) = 0.
\]
By Lemma~\ref{L:drinfeld schemes pro-p} again, the first term in this sequence vanishes; we thus deduce that
\begin{equation} \label{eq:from DL}
H^1(\varphi, L \otimes_{\FF_q} F)^{\varphi_L}  = 0.
\end{equation}

Note that the action of $\varphi-1$ on $R$ preserves the decomposition $R \cong \bigoplus_{e_1,e_2} R_{e_1, e_2}$.
Consequently, it suffices to check that $H^1(\varphi, R_{e_1, e_2})^{\varphi_L} = 0$ for all $e_1, e_2$.

First, suppose that $(e_1, e_2) = (0,+)$.
Any class in $H^1(\varphi, R_{0,+})$ may be represented  by an element $y_{0,+} \in R_{0,+}$ which is a convergent sum of products supported on $\{0\} \times [1, \infty)$: namely, all but finitely many terms already have the right support, and each remaining term (which initially is only constrained to have support in $\{0\} \times [c, \infty)$ for some $c>0$ depending on the term) can be replaced by its image under a suitable power of Frobenius to fix its support. However, we can now replace
$y_{0,+}$ with 
\[
y_{0,+} + z^p - z, \qquad z = \sum_{m=0}^\infty y_{0,+}^{p^m}
\]
to see that $y_{0,+}$ represents the zero class in $H^1(\varphi, R_{0,+})$.
We deduce that $H^1(\varphi, R_{0,+}) = 0$ even before taking $\varphi_L$-invariants.
Similar considerations apply in the cases $(e_1, e_2) = (+,0), (+,+)$.

Next, suppose that $(e_1, e_2) = (0,-)$. Any class in $H^1(\varphi, R_{0,+})^{\varphi_L}$ may be represented  by an element of $R_{0,+}$ which is a \emph{finite} sum of products supported on $\{0\} \times (-\infty, 0)$;
in particular, this yields an element of $L \otimes_{\FF_q} F$ whose image in $H^1(\varphi, L \otimes_{\FF_q} F)$
is again $\varphi_L$-invariant. By \eqref{eq:from DL}, this image vanishes, from which we deduce that
$H^1(\varphi, R_{0,-})^{\varphi_L} = 0$. Similar considerations apply in the cases
$(e_1, e_2) = (-, 0), (-,-), (0,0)$.

Finally, suppose that $(e_1, e_2) = (+,-)$. Any class in $H^1(\varphi, R_{+,-})^{\varphi_L}$
may be represented by an element $y_{+,-}$ of $R_{+,-}$ which is a convergent sum of products supported on
$[c_1, \infty) \times [c_2, 0)$ for some $c_1 > 0, c_2 < 0$ depending on $y_{+,-}$. Consider the formal expansion $\sum_{i,j} y_{i,j} u^i t^j$ of $y_{+,-}$;
since the resulting class in $H^1(\varphi, R_{+,-})$ is $\varphi_L$-invariant,
for every $i,j$ and every nonnegative integer $m$ we have
\[
\sum_{n \in \ZZ} y_{ip^n, jp^n}^{p^{-n}} = \sum_{n \in \ZZ} \varphi^m_{\kappa_L}(y_{ip^{n},jp^{n+m}}^{p^{-n-m}}).
\]
For $m$ sufficiently large (depending on $i,j$), the conditions $ip^n \geq c_1$ and $j p^{n+m} \geq c_2q$
on $n$ are incompatible, so the sum on the right is identically zero. 
From this, it follows that the formal sum $\sum_{n=0}^\infty y_{+,-}^{p^n}$ converges to an element of $R_{+,-}$, and hence that the class of $y_{+,-}$ in $H^1(\varphi, R_{+,-})$ vanishes. Similar considerations apply in the case $(e_1, e_2) = (-,+)$.
\end{proof}

\section{Inputs from $p$-adic differential equations}
\label{sec:pde}

We next bring in some relevant input from the theory of $p$-adic differential equations.

\begin{hypothesis}
Throughout \S\ref{sec:pde}, assume that $E$ is of characteristic 0.
Let $K_0$ be the completion of $F(T)$ for the Gauss norm.
(This field often called the \emph{field of analytic elements} over $F$ in the variable $T$; this terminology is due to Krasner \cite{krasner1, krasner2}.)
Let $K$ be a finite tamely ramified extension of $K_0$. (Note that references into \cite{kedlaya-book} only cover the case $K = K_0$;
see \cite[\S 2.2]{kedlaya-conv} for adaptations to the general case.)
\end{hypothesis}

\begin{defn} \label{D:subsidiary radii}
The derivation $\frac{d}{dT}$ on $F(T)$ is submetric for the Gauss norm, so it extends continuously to a derivation on $K$ with operator norm 1. By a \emph{differential module} over $K$, we will mean a finite-dimensional $K$-vector space $V$ equipped with a derivation $D$ satisfying the Leibniz rule with respect to $\frac{d}{dT}$. The example to keep in mind is a finite \'etale $K$-algebra equipped with the unique $K$-linear derivation extending $\frac{d}{dT}$.

Let $V$ be a differential module over $K$ of dimension $n$.
We define the \emph{subsidiary radii} of $V$ as in \cite[Definition~9.8.1]{kedlaya-book};
this is a multisubset of $(0,1]$ of cardinality $n$, and is invariant under base extension
\cite[Proposition~10.6.6]{kedlaya-book}. Geometrically, the subsidiary radii may be interpreted as the radii of convergence of local horizontal sections of $V$ in a generic unit disc \cite[Theorem~11.9.2]{kedlaya-book}.

Let $e^{-f_1(V)}, \dots, e^{-f_n(V)}$ be the subsidiary radii of $V$
listed in ascending order (with multiplicity),
and define
\[
F_i(V) := f_1(V) + \cdots + f_i(V) \qquad (i=1,\dots,n).
\]
\end{defn}

\begin{remark} \label{R:change of parameter}
Suppose that $K$ is also a finite tamely ramified extension of the completion of $F(T')$
for the Gauss norm. Then any differential module $V$ with respect to $\frac{d}{dT}$ is also
a differential module with respect to $\frac{d}{dT'}$, and the two resulting definitions of subsidiary radii coincide. This follows from the geometric interpretation in terms of convergence in a generic disc.
\end{remark}

\begin{remark}
One can generally predict properties of the subsidiary radii of a differential module $V$
by modeling them by the inverse norms of the eigenvalues of some linear transformation $T_V$ on some $n$-dimensional vector space over $K$, subject to the functoriality properties
\[
T_{V_1 \oplus V_2} = T_{V_1} \oplus T_{V_2}, \qquad
T_{V_1 \otimes V_2} = T_{V_1} \otimes 1_{V_2} + 1_{V_1} \otimes T_{V_2}. 
\]
\end{remark}

\begin{lemma}[Christol--Dwork] \label{L:christol-dwork}
Let $V$ be a differential module over $K$ of rank $n$.
Let $\bv \in V$ be a \emph{cyclic vector}, i.e., 
an element such that $\bv, D(\bv),\dots, D^{n-1}(\bv)$ form a basis of $V$.
(Such elements always exist; see for example \cite[Theorem~5.4.2]{kedlaya-book}.)
Write $D^n(\bv) = a_0 + a_1 \bv + \cdots + a_{n-1} D^{n-1}(\bv)$ with $a_0,\dots,a_{n-1} \in K$.
Form the multiset consisting of $p^{-1/(p-1)} \left| \lambda \right|^{-1}$ as $\lambda$ varies over the roots
of the polynomial $T^n - a_{n-1} T^{n-1} - \cdots - a_0$ in an algebraic closure of $K$.
Then this multiset coincides with the subsidiary radii of $V$ in its values less than $p^{-1/(p-1)}$.
\end{lemma}
\begin{proof}
The original reference is \cite[Th\'eor\`eme~1.5]{christol-dwork}; see also \cite[Theorem~6.5.3]{kedlaya-book}.
\end{proof}

\begin{lemma} \label{L:pushforward}
Let $V$ be a differential module over $K$ of rank $n$.
Let $V'$ be the restriction of scalars of $V$ along the unique continuous $F$-linear homomorphism $K \to K$ taking $T$ to $T^p$.
View $V'$ as a differential module over $K$ of rank $pn$ using the derivation $D' = pT^{p-1} D$
(this is called the \emph{Frobenius pushforward} of $V$). Then the subsidiary radii of $V'$ coincide with the multiset consisting of
\[
\begin{cases} \{\rho^p\} \bigcup \left( \{p^{-p/(p-1)} \mbox{\,($p-1$ times)}\}  \right)& \rho > p^{-1/(p-1)} \\
\{p^{-1} \rho \mbox{\,($p$ times)}\} & \rho \leq p^{-1/(p-1)}
\end{cases}
\]
for $\rho$ running over the subsidiary radii of $V$.
\end{lemma}
\begin{proof}
See \cite[Theorem~10.5.1]{kedlaya-book}.
\end{proof}

\begin{lemma} \label{L:end}
Let $V$ be a differential module over $K$ of rank $n$.
Suppose that the subsidiary radii of $V$ are all equal to some value $\rho<1$.
Then at least $n$ of the subsidiary radii of $V^\vee \otimes V$ are strictly greater than $\rho$.
\end{lemma}
\begin{proof}
This follows from the existence of a \emph{refined spectral decomposition} of $V$ over a suitable tamely ramified extension of $K$; see \cite[Theorem~10.6.2, Theorem~10.6.7]{kedlaya-book}.
\end{proof}

\begin{hypothesis}
For the remainder of \S\ref{sec:pde},
let $D$ be a one-dimensional smooth affinoid space over $F$ (viewed as an adic space).
Let $\calE$ be a vector bundle over $D$ of rank $n$ equipped with an $F$-linear connection.
\end{hypothesis}

\begin{defn} \label{D:residue field at point}
Let $x \in D$ be a point of type 2 or 5 in the sense of Definition~\ref{D:types of points}. 
Then the residue field $\calH(x)$ may be viewed as a 
finite tamely ramified extension of $K_0$ in some fashion (e.g., see the proof of \cite[Theorem~5.3.6]{kedlaya-conv}). We may thus define the quantities
$f_i(\calE, x), F_i(\calE,x)$ for $i=1,\dots,n$ by viewing the fiber $\calE_x$ as a differential module over $\calH(x)$. By Remark~\ref{R:change of parameter}, this definition does not depend on auxiliary choices.

Now let $x \in D$ be a point of type 3 or 4. Then after replacing $F$ with a suitable extension field, we may lift $x$ to a point of type 2 and apply the previous paragraph to define $f_i(\calE, x), F_i(\calE, x)$.
Again using Remark~\ref{R:change of parameter}, we may see that this definition does not depend on auxiliary choices.
\end{defn}

\begin{remark}
The previous definition can again be interpreted in terms of the convergence of local horizontal sections.
See \cite{kedlaya-simons} for an overview of the results one obtains in this manner.
\end{remark}

\begin{lemma} \label{L:transfer}
Let $U$ be an open disc in $D$ bounded by $x$.
If $f_1(\calE,x) = 0$, then the restriction of $\calE$ to $U$ has trivial connection.
\end{lemma}
\begin{proof}
This is an instance of the Dwork transfer principle \cite[Theorem~9.6.1]{kedlaya-book}.
\end{proof}

\begin{lemma} \label{L:annulus}
Let $U$ be an open annulus in $D$ and let $T$ be a coordinate on $U$.
For each positive integer $m$, let $U_m$ be the $m$-fold cover of $U$ with coordinate $T^{1/m}$.
Suppose that there exists a finite \'etale cover $\pi: D' \to D$ such that the pullback of $\calE$ to $D'$ has trivial connection. 
If $f_1(\calE,x) =0$ for each $x$ in the skeleton of $U$, then there exists a positive integer $m \leq \deg(\pi)!$ not divisible by $p$ such that the pullback of $\calE$ to $U_m$ has trivial connection.
\end{lemma}
\begin{proof}
We first note that thanks to \cite[Corollary~13.6.4]{kedlaya-book}, we may deduce the desired result from the corresponding statement about some smaller open annulus contained in $U$. We
are thus free to shrink $U$ as needed throughout the argument. To begin with, we may assume that $\pi^{-1}(U)$ admits a connected component $U'$ which is itself an annulus.

Let $\calF$ be the pushforward along $\pi$ of the trivial connection on $U'$.
Then $\calF$ is semisimple and $\calE$ occurs as a subobject, so $\calE$ is also semisimple.

The condition that $f_1(\calE, x) = 0$ for each $x$ in the skeleton of $U$ means, in classical language, that
$\calE$ satisfies the \emph{Robba condition} on $U$. We may thus apply the Christol--Mebkhout theory of \emph{$p$-adic exponents}
as described in \cite[Chapter~13]{kedlaya-book}.

Let $\calR^{\bd}$ be the ring of germs of bounded analytic functions on subannuli of $U$ with outer radius 1 (i.e.,
the \emph{bounded Robba ring} over $F$ in the sense of \cite[Definition~15.1.2]{kedlaya-book}); this is a henselian discrete valuation ring with residue field
$\kappa_F((t))$ where $\kappa_F$ is the residue field of $F$.
The corresponding ring associated to $U'$ is a finite extension $\calR^{\bd \prime}$ of $\calR^{\bd}$.

By hypothesis, $\calE \otimes \calR^{\bd \prime}$ carries the trivial connection; we claim that this remains true if we replace $\calR^{\bd \prime}$ with its maximal unramified subextension.
To this end, it suffices (by the Galois theory of extensions of henselian fields; see for example \cite[Chapter~3]{kedlaya-book})
 to check that if $\calR^{\bd \prime}$ is a cyclic totally ramified extension of prime order, then $\calE$ already carries the trivial connection. If this extension is tamely ramified,
then by Abhyankar's lemma it is obtained by making an extension of $F$ and the claim is evident. If it is wildly ramified, then it is cyclic of order $p$
and the restriction of scalars of $\calF$ splits as a direct sum of copies of $\calE$ twisted by the powers of a nontrivial rank-1 connection $\calL$ whose
$p$-th power is trivial. If $\calE$ is not itself trivial, then for some $i \in \{1,\dots,p-1\}$, $\calL^{\otimes i}$ occurs as a subquotient of the tensor product of two connections
satisfying the Robba condition, so $\calL$ itself satisfies the Robba condition.
We may thus apply \cite[Theorem~13.5.5]{kedlaya-book} to choose an exponent $B \in \ZZ_p$ for $\calL$.
Since $pB$ is then an exponent for $\calL^{\otimes p}$ whose connection is trivial, we may apply \cite[Theorem~13.5.6]{kedlaya-book}
to deduce that $pB \in \ZZ$. It follows that $B \in \ZZ$ and so $\calL$ is itself trivial; this yields a contradiction which proves the claim.

We may now assume that $\calR^{\bd \prime}$ is an unramified extension of $\calR^{\bd}$, so that $\calE$
corresponds to a representation of the Galois group of $\kappa_F((t))$. We may now apply  
\cite[Theorem~19.4.1]{kedlaya-book} to see that this representation is itself at most tamely ramified.
By Abhyankar's lemma again, this representation becomes unramified if we pull back from $U$ to $U_m$
for some $m \leq \deg(\pi)!$ not divisible by $p$; this yields the desired result.
\end{proof}

\section{Relative connections}
\label{sec:rel con FF}

We now extend the discussion of $p$-adic connections to a relative setting, with the key case being when the base field is replaced by a Fargues--Fontaine curve.

\begin{hypothesis}
Throughout \S\ref{sec:rel con FF}, let $A$ be a uniform Huber ring over $\QQ_p$.
\end{hypothesis}

\begin{defn}
Let $\calM(A)$ denote the Gel'fand spectrum of $A$ in the sense of Berkovich, i.e., the space of all
bounded multiplicative seminorms on $A$, equipped with the evaluation topology;
it may be identified with the maximal Hausdorff quotient of the adic spectrum $\Spa(A,A^\circ)$. 
Since $A$ is assumed to be uniform, the supremum over $\calM(A)$ is a norm on $A$ which induces the correct topology on $A$.
For $\alpha \in \calM(A)$, let $\ker(\alpha) := \alpha^{-1}(0)$ be the inverse image of $0$ under $\alpha$,
and let $\calH(\alpha) := \widehat{A/\ker(\alpha)}$ denote the (completed) residue field of $\alpha$.

For $\alpha \in \calM(A)$, we say that a rational localization $A\to B$ (in the sense of Huber rings)
\emph{encircles} $\alpha$ if the map $\calM(B) \to \calM(A)$ identifies $\calM(B)$ with a neighborhood of $\alpha$ in $\calM(A)$. Such neighborhoods form a neighborhood basis of $\alpha$ in $\calM(A)$.
\end{defn}

\begin{defn} \label{D:analytic elements}
Equip the ring $A \langle T \rangle$ with the Gauss extension of the norm on $A$.
For each $\alpha \in \calM(A)$, let $\tilde{\alpha} \in \calM(A \langle T \rangle)$ be its Gauss extension.
Let $S$ be the multiplicative subset of $s \in A \langle T \rangle$
for which $\tilde{\alpha}(s) \neq 0$ for all $\alpha \in \calM(A)$
(equivalently, the coefficients of $s$ have no common zero in $\calM(A)$).
We may then form a new Huber ring $R_A$ by completing the algebraic localization $A \langle T \rangle_S$
for the supremum of the seminorms induced by $\tilde{\alpha}$ for each $\alpha \in \calM(A)$.
By construction, $R_A$ is again a uniform Huber ring; 
in the case where $A = K$ is a nonarchimedean field with norm $\alpha$, the ring $R_K$ is simply the completion of $K(T)$ with respect to the multiplicative norm $\tilde{\alpha}$.
Note that any homomorphism $A \to B$ of uniform Huber rings induces a homomorphism $R_A \to R_B$.
\end{defn}

\begin{defn}
By a \emph{differential module} over $R_A$, we will mean a finite
projective $R_A$-module $M$ equipped with a derivation $D$ satisfying the Leibniz rule with respect to $\frac{d}{dT}$. We define the \emph{subsidiary radii} of $M$ at $\alpha \in \calM(A)$
by base extension from $R_A$ to $R_{\calH(\alpha)}$ and application of
Definition~\ref{D:subsidiary radii}.
For $\alpha \in \calM(A)$
let $e^{-f_1(M, \alpha)}, \dots, e^{-f_n(M,\alpha)}$ be the subsidiary radii of $M$ at $\alpha$
listed in ascending order (with multiplicity),
and define
\[
F_i(M, \alpha) := f_1(M, \alpha) + \cdots + f_i(M, \alpha) \qquad (i=1,\dots,n).
\]
\end{defn}

\begin{lemma} \label{L:convex}
Let $I$ be a closed subinterval of $(0, \infty)$.
Let $M$ be a differential module over $R_A$ of constant rank $n$ for $A = B^I_{L,E} \widehat{\otimes}_{W(\overline{\FF}_q)_E} \frako_F$. 
Then for $i=1,\dots,n$, the function
\[
t \mapsto F_i(M, \lambda_t)
\]
on $I$ is convex in $t$.
\end{lemma}
\begin{proof}
It suffices to check that for any $\epsilon > 0$, the function
\[
t \mapsto \sum_{j=1}^i \max\{f_j(M, \lambda_t), \epsilon\}
\]
is convex.
It further suffices to check this claim locally around some interior point $t \in I$.

We first check the claim for $\epsilon = \frac{1}{p-1} \log p$.
By arguing as in \cite[Lemma~11.5.1]{kedlaya-book}, we may construct a basis $v_1,\dots,v_n$ of
$M \otimes_{R_A} R_{B^{J}_{L,E} \widehat{\otimes}_{W(\overline{\FF}_q)_E} \frako_F}$
for some closed subinterval $J$ of $I$ containing $t$ in its interior,
such that the values of $F_i(M, \lambda_u)$ for $u \in J$ can be read off from the characteristic
polynomial of the matrix of action of $\frac{d}{dT}$ on this basis.
We may then deduce the claim from the convexity of $t \mapsto \log \lambda_t(x)$
(see Definition~\ref{D:norms}).

To conclude, it suffices to check the claim for $\epsilon = \frac{p^{-m}}{p-1} \log p$
for each nonnegative integer $m$. For this, we apply the previous argument to treat the base case $m=0$
and Lemma~\ref{L:pushforward} to handle the induction step.
\end{proof}

\begin{defn}
Let $X$ be an adic space over some nonarchimedean field $K$ over $\QQ_p$.
Let $D$ be a subset of the analytic affine line over $K$ containing the Gauss point (the boundary of the closed unit disc).
Let $\calE$ be a vector bundle of constant rank $n$
on $X \times_K D$ equipped with an $\calO_X$-linear connection.
By base extension to Definition~\ref{D:subsidiary radii}, 
we may define the functions $f_i(\calE, x)$ and $F_i(\calE,x)$ for $x \in X$ and $i \in \{1,\dots,n\}$.
\end{defn}

\begin{hypothesis}
For the remainder of \S\ref{sec:rel con FF},
let $D$ be a connected rational subspace of the analytic affine line over $F$ containing the Gauss point,
and let $\calE$ be a vector bundle of constant rank $n$
on $X_{L,E} \times_E D$ equipped with an $\calO_{X_{L,E} \times_E F}$-linear connection.
Let $\tilde{\calE}$ be the pullback of $\calE$ to $Y_{L,E} \times_E D$.
\end{hypothesis}

\begin{lemma} \label{L:sup const}
For $i=1,\dots,n$, the function $x \mapsto f_i(\calE, x)$ on $Y_{L,E} \times_E F$ is constant.
We will hereafter denote by $f_i(\calE)$ this constant value.
\end{lemma}
\begin{proof}
We proceed by strong induction on $i$.
We start by using Remark~\ref{R:FF connected} to make an identification of topological spaces
\[
\{\lambda_t: t > 0\} \times_E F \cong 
(0, +\infty) \times \widehat{\ZZ}.
\]
We then note that
\[
t \mapsto \sup\{F_i(\tilde{\calE}, x): x \in \{t\}  \times \widehat{\ZZ} \}
\]
is both convex (by Lemma~\ref{L:convex})
and invariant under $t \mapsto tq$; it must therefore be constant.
Denote by $c$ the constant value.

Our next goal is to check that $F_i(\tilde{\calE},x)$ is in fact constant on $(0, +\infty) \times \widehat{\ZZ}$.
This is guaranteed in case $c$ equals the known constant value of $F_{i-1}(\tilde{\calE},x)$, as this is a lower bound on $F_i(\tilde{\calE},x)$; we may thus assume that the difference 
\[
c' = \sup\{f_i(\tilde{\calE},x): x \in \{t\} \times \widehat{\ZZ} \}
\]
is positive (and again independent of $t$).

From the proof of Lemma~\ref{L:convex} (and using that $c'>0$), we see that for any fixed $t$, the set $U_t$ of $y \in \widehat{\ZZ}$ for which
$x = (t, y)$ maximizes $F_i(\tilde{\calE}, x)$ is nonempty and open in $\widehat{\ZZ}$. 
Choose a set $U$ of the form $U_{t_0}$ for some $t_0 > 0$. 
For $y \in U$, the function $t \mapsto F_i(\tilde{\calE},(t,y))$ on $(0, +\infty)$ is convex (again by Lemma~\ref{L:convex}), has $c$ as an upper bound everywhere, and achieves this bound at one point; it therefore is identically equal to $c$.

By Remark~\ref{R:FF connected} again, 
for $y$ belonging to any translate of $U$ under the action of $\ZZ \subset \widehat{\ZZ}$,
 $F_i(\tilde{\calE}, (t,y)) = c$ for all $t>0$.
Since $U$ is nonempty and open, its translates cover $\widehat{\ZZ}$;
hence $F_i(\tilde{\calE}, x) = c$ for all $x \in (0, +\infty) \times \widehat{\ZZ}$.

At this point, it now suffices to check that the function
$\alpha \mapsto F_i(\tilde{\calE}, \alpha)$ on 
$\calM(B^{[t,t]}_{L,E} \widehat{\otimes}_{W(\overline{\FF}_q)_E} \frako_F)$ is constant.
As in the proof of Lemma~\ref{L:convex}, it suffices to check that for any $\epsilon > 0$, the function
\[
\alpha \mapsto \sum_{j=1}^i \max\{f_j(\tilde{\calE}, \alpha), \epsilon\}
\]
is constant; again using Lemma~\ref{L:pushforward}, we may further reduce to the case $\epsilon = \frac{1}{p-1} \log p$. To check this case, set notation as in the proof of Lemma~\ref{L:convex}.
Let $a_0,\dots,a_{n-1}$ be the coefficients of the characteristic polynomial of the matrix of action of
$\frac{d}{dT}$ on the basis $v_1,\dots,v_n$. Carrying the equality $F_i(\tilde{\calE}, (u,y)) = c$
(for $u$ in some neighborhood of $t$) back through the proof of Lemma~\ref{L:convex},
we see that if $a_i$ contributes a vertex to the Newton polygon for $\alpha = \lambda_t$, then the position of that vertex varies linearly (in a neighborhood of $t$).
This in turn implies that the image of $a_i$ in $\Gr R_{B^{[t,t]}_{L,E} \widehat{\otimes}_{W(\overline{\FF}_q)_E} \frako_F}$ 
belongs to $\Gr R_F$ (compare \cite[Theorem~11.2.1(c)]{kedlaya-book}).

To conclude, note that as we vary $\alpha \in \calM(B^{[t,t]}_{L,E} \widehat{\otimes}_{W(\overline{\FF}_q)_E} \frako_F)$, the spectral norm of $\frac{d}{dT}$ on $R_{\calH(\alpha)}$
does not change. Consequently, as in the proof of Lemma~\ref{L:convex}, we may continue to argue as in \cite[Lemma~11.5.1]{kedlaya-book}
to read off $\sum_{j=1}^i \max\{f_j(\tilde{\calE}, \alpha), \epsilon\}$ from $a_1,\dots,a_n$ just as we did for $\alpha = \lambda_t$;
this yields the desired result.
(By contrast, the corresponding argument in the context of \cite{kedlaya-book} is more subtle due to the variation of the spectral norm of $\frac{d}{dT}$; see for example \cite[Lemma~4.3.12]{kedlaya-book}.)
\end{proof}

\begin{lemma} \label{L:decomposition}
After possibly shrinking $D$, there exists a unique decomposition $\calE = \calE_1 \oplus \cdots \oplus \calE_k$
of vector bundles with connection (with nonzero summands) with the following properties.
\begin{enumerate}
\item[(i)]
For $i=1,\dots,k$, we have $f_1(\calE_i) = \cdots = f_{\rank(\calE_i)}(\calE_i)$.
\item[(ii)]
We have $f_1(\calE_1) > \cdots > f_1(\calE_k)$.
\end{enumerate}
\end{lemma}
\begin{proof}
We may assume that $n > 0$.
Let $j$ be the largest integer such that $f_j(\calE) = f_1(\calE)$; it suffices to split $\calE$
as a direct sum $\calE_1 \oplus \calE_2$ with $\rank(\calE_1) = j$ such that $f_i(\calE_1) = f_1(\calE)$
for all $i$ and $f_i(\calE_2) < f_1(\calE)$ for all $i$. For this, we may assume that $j<n$,
and use Frobenius pushforwards
(as in Lemma~\ref{L:pushforward}, but see more specifically the proof of \cite[Theorem~12.2.2]{kedlaya-book}) to reduce to the case
where $f_1(\calE) > p^{-1/(p-1)}$. In this case,
set notation as in the proof of Lemma~\ref{L:sup const}; we may then deduce the claim by applying 
a suitable version of Hensel's lemma, such as \cite[Theorem~2.2.2]{kedlaya-book}, to the polynomial coming from the cyclic vector.
\end{proof}

\begin{remark}
Consider an open disc in the analytic affine line over $F$ bounded by the Gauss point. If the entire disc is contained in $D$, then the conclusion of Lemma~\ref{L:decomposition} holds without removing any points of the disc, again as in the proof of \cite[Theorem~12.2.2]{kedlaya-book}. By contrast, if some of the disc is missing from $D$, then more of it may have to be removed in order to achieve the conclusion of  Lemma~\ref{L:decomposition}.
\end{remark}

\section{Elimination of one parameter}
\label{sec:elimination}

\begin{hypothesis} \label{H:elimination}
Throughout \S\ref{sec:elimination},
suppose that every finite \'etale cover of $X_{L,E} \times_E F$ splits completely.
Let $D$ be a one-dimensional smooth affinoid space over $F$.
Let $x \in D$ be a point of type 2, 3, 4 in the sense of Definition~\ref{D:types of points};
in the type 2 case, assume further that $D$ is a strict neighborhood of $x$.

Let $\rho: \pi_1^{\prof}(X_{L,E} \times_E D) \to \GL(V)$ be a discrete representation
on a finite-dimensional $F$-vector space $V$.
We will make various statements that hold not for $\rho$ itself, but its restriction
to $\pi_1^{\prof}(X_{L,E} \times D')$ for some connected \'etale neighborhood $D'$ of $x$ in $D$
(which in the type 2 case must again be a strict neighborhood of some preimage of $x$).
To indicate this restriction, we will say that such statements hold \emph{after replacing $D$}.
\end{hypothesis}

\begin{defn} \label{D:rep to bundle}
Let $f: U \to X_{L,E} \times_E D$ be a finite \'etale cover such that the restriction of $\rho$
to $\pi_1^{\prof}(U)$ splits completely.
We then obtain a diagonal action of $\pi_1^{\prof}(X_{L,E} \times_E D, \overline{x})$
on $f_* \calO_U \otimes_F V$; let $\calE_{\rho}$ be the fixed submodule for this action,
viewed as a vector bundle on $X_{L,E} \times_E D$.
The connection on $\calE_{\rho}$ is induced by differentiation on $\calO_U$ with respect to $X_{L,E}$.

For $n = \dim(V)$, for $y \in D$ of type 2 or 5 with residual genus 0, we may define quantities
$f_1(\calE_{\rho}, y), \dots, f_n(\calE_{\rho}, y)$ using Lemma~\ref{L:sup const};
for $y$ of type 3 or 4, we may define these quantities by enlarging $F$ to lift $y$ to a point of type 2 or 5
(which will necessarily have residual genus 0) and then applying
Lemma~\ref{L:sup const} after suitable rescaling. 
(It is possible to extend the construction to points of type 2 or 5 of positive genus, but this is not crucial here.)
\end{defn}

\begin{lemma} \label{L:zero splits}
Suppose that $f_i(\calE_\rho, y) = 0$ for all $i$ and $y$. Then $\rho$ becomes trivial after replacing $D$.
\end{lemma}
\begin{proof}
We first verify that  for each geometric point $z$ of $X_{L,E} \times_E F$,
after replacing $D$,
the restriction of $\rho$ to $\pi_1^{\prof}(z \times_F D)$ 
becomes trivial.
\begin{itemize}
\item
If $x$ is of type 3, then after replacing $D$, we may assume that $D$ is an annulus.
By Lemma~\ref{L:annulus}, the restriction of $\rho$ to $\pi_1^{\prof}(z \times_F D)$
becomes trivial after replacing $D$.
\item
If $x$ is of type 4, then after replacing $D$, we may assume that $D$ is a disc.
By Lemma~\ref{L:transfer}, the restriction of $\rho$
to $\pi_1^{\prof}(z \times_F D)$ is trivial.
\item
If $x$ is of type 2, then for all but finitely many specializations $x'$ of $x$, we may find
an open disc $W$ in $D$ bounded by $x'$; by Lemma~\ref{L:transfer}, the restriction of $\rho$
to $\pi_1^{\prof}(z \times_F W)$ is trivial. For each of the other remaining specializations $x'$,
we may find an open annulus $W$ in $D$ bounded by $x'$ at one end; 
by Lemma~\ref{L:annulus}, the restriction of $\rho$ to $\pi_1^{\prof}(z \times_F W)$ becomes trivial after replacing $D$. Combining these results, we see that after replacing $D$, the restriction of $\rho$
to $\pi_1^{\prof}(z \times_F D)$ factors through $\pi_1^{\prof}(C_\ell)$ where $C$ is the residual curve of $D$ at $x$
and $\ell$ is the residue field of $\calH(z)$. Since $C$ is proper by hypothesis, we may apply \cite[Corollary~4.1.19]{kedlaya-aws}
to deduce that $\pi_1^{\prof}(C_\ell) \to \pi_1^{\prof}(C)$ is a homeomorphism; consequently, after replacing $D$, the restriction of $\rho$ to $\pi_1^{\prof}(z \times_F D)$ becomes trivial.
\end{itemize}
For any given $z$, it formally follows that after replacing $D$,
there exists some neighborhood $U$ of $z$ in $X_{L,E} \times_E F$
such that the restriction of $\rho$ to $\pi_1^{\prof}(U \times_F D)$ 
factors through $\pi_1^{\prof}(U)$.
By compactness, we may choose $D$ uniformly over some finite set of geometric points $z$ for which the neighborhoods $U$ form a covering of $X_{L,E}$; we then deduce that (after replacing $D$)
the restriction of $\rho$ to $\pi_1^{\prof}(X_{L,E} \times_E D)$ 
factors through $\pi_1^{\prof}(X_{L,E} \times_E F)$.
As we are working under Hypothesis~\ref{H:elimination}, the latter group is trivial; this proves the claim.
\end{proof}

\begin{lemma} \label{L:full splits}
Under no additional hypotheses, $\rho$ becomes trivial after replacing $D$.
\end{lemma}
\begin{proof}
Assume by way of contradiction that the conclusion fails for some $\rho$.
We may assume that $\rho$ is irreducible with nontrivial simple image $G$,
and remains so after replacing $D$ arbitrarily.

Suppose first that $G$ is abelian, and hence cyclic of prime order. 
If this order is prime to $p$, we obtain a contradiction
against Lemma~\ref{L:zero splits}; if the order is $p$, we obtain a contradiction against
Lemma~\ref{L:no abelian cover}.

Suppose next that $G$ is nonabelian.
Let $[\rho]$ be the Tannakian category generated by $\rho$ (keeping in mind that $F$ is algebraically closed of characteristic $0$). Since $\rho$ has simple image, every nontrivial irreducible $\tau \in [\rho]$ is a generator.
For each $\tau$, form the vector bundle $\calE_\tau$ as in Definition~\ref{D:rep to bundle}
and consider the functions $f_i(\calE_\tau, y)$ for $i=1,\dots,\dim(\tau)$
for $y \in D$ of type 2 with residual genus 0, or of type 3 or 4. In particular, these functions are well-defined for $y$ in the skeleton of any neighborhood of $x$ in $D$, and so it is meaningful to consider limiting behavior as $y \to x$.

For clarity, let us first treat the case where $x$ itself is not a point of type 2 with positive residual genus, so that we may even take $y = x$. We divide the argument into two parallel steps.
\begin{enumerate}
\item[(i)]
Suppose that $f_1(\calE_\tau, x) > 0$.
By Lemma~\ref{L:decomposition} and the irreducibility of $\tau$, 
$f_1(\calE_\tau, x) = \cdots = f_{\dim(\tau)}(\calE_\tau, x)$.
Moreover, since each $\tau$ generates $[\rho]$, $f_1(\calE_\tau, x)$ does not depend on $\tau$ either;
since the trivial representation occurs in $\tau^\vee \otimes \tau$ with multiplicity 1,
\[
f_i(\calE_{\tau^\vee \otimes \tau}, x) = \begin{cases} f_1(\calE_\tau,x) & i = 1,\dots,\dim(\tau)^2-1 \\
0 & i = \dim(\tau)^2. \end{cases}
\]
On the other hand, if $f_1(\calE_\tau, x) > 0$, then Lemma~\ref{L:end} implies that
$f_i(\calE_{\tau^\vee \otimes \tau}, x) = 0$ for $i = \dim(\tau)^2 - \dim(\tau) + 1$;
this forces $\dim(\tau) = 1$, contradicting our assumption that $G$ is nonabelian.

\item[(ii)]
Suppose that $f_1(\calE_\tau, x) = 0$.
By \cite[Theorem~11.3.2]{kedlaya-book}, each function $f_i(\calE_\tau, y)$ restricts to a linear function
on the skeleton of some neighborhood of $x$ in $D$.
By Lemma~\ref{L:decomposition} again, 
$f_1(\calE_\tau, y) = \cdots = f_{\dim(\tau)}(\calE_\tau, y)$ on some skeleton,
and this common value does not depend on $\tau$ either.
Applying Lemma~\ref{L:end} as in the previous paragraph yields a contradiction unless
$f_1(\calE_\tau,y)$ is identically zero on some skeleton. But in this case,
Lemma~\ref{L:zero splits} yields a final contradiction.
\end{enumerate}

To complete the proof, we must revise the last two paragraphs to cover the case where ($G$ is nonabelian and) $x$ is of type 2 with positive residual genus. 
The approach will be to replace $f_i(\calE_\tau, x)$ with $\lim_{y \to x} f_i(\calE_\tau, y)$;
by the continuity aspect of \cite[Theorem 5.3.8]{kedlaya-conv}, these limits
exist and compute subsidiary radii in the same fashion as in \S\ref{sec:rel con FF}.
In particular, Lemma~\ref{L:end} applies without issue
(in particular, the analogue of Lemma~\ref{L:sup const} holds).

If $\lim_{y \to x} f_1(\calE_\tau,x) = 0$, then step (ii) applies without change.
If $\lim_{y \to x} f_1(\calE_\tau,x) > 0$, then step (i) applies \emph{mutatis mutandis}
except for the application of Lemma~\ref{L:decomposition}; however, we may take care of this step by pushing forward along a finite morphism from $D$ to a subspace of the affine line
as in the proof of \cite[Theorem~5.3.6]{kedlaya-conv}. (This does not preserve simplicity of the image of the representation, but that is not material to Lemma~\ref{L:decomposition}.)
\end{proof}

\section{Simple connectivity and Drinfeld's lemma}

We now put everything together to prove 
Theorem~\ref{T:simple connectivity}, then deduce a form of Drinfeld's lemma for perfectoid spaces.
\begin{theorem} \label{T:simple connectivity2}
Every finite \'etale cover of $X_{L,E} \times_E F$ splits completely. In particular,
Theorem~\ref{T:simple connectivity} holds.
\end{theorem}
\begin{proof}
By Remark~\ref{R:interval perfectoid}, we may assume that $E$ is of characteristic $0$.
By Lemma~\ref{L:splitting cover discrete}, the claim holds in case $F$ is a completed algebraic closure of $E$.
By transfinite induction, it suffices to prove that if the claim holds for $F$, then it holds for the completed algebraic closure $F'$ of the completion of $F(T)$ for some multiplicative norm.

Note that the norm determines a point $x$ of the adic affine line over $F$ of type 2, 3, or 4;
the point $x$ admits a fundamental system of \'etale neighborhoods $D$, each a smooth one-dimensional affinoid space over $F$. Thanks to the henselian property of local rings of adic spaces \cite[Lemma~2.4.17]{kedlaya-liu1} and the fact that $X_{L,E}$ is qcqs,
any finite \'etale cover of $X_{L,E} \times_E F'$ can be spread out to $X_{L,E} \times_E D$ for some $D$.
We may thus apply Lemma~\ref{L:full splits} to deduce that the cover can be split by replacing $D$ with another \'etale neighborhood $D'$; this implies the desired conclusion.
\end{proof}

\begin{defn}
Let $\Pfd$ denote the category of perfectoid spaces in characteristic $p$.
\end{defn}

We must take note of a slight difference in behavior between absolute products of schemes and of perfectoid spaces.
This expands upon Remark~\ref{R:diamond interpretation}.

\begin{lemma} \label{L:Pfd products}
Let $X_1,\dots,X_n \in \Pfd$ be arbitrary.
\begin{enumerate}
\item[(a)]
The (absolute) product $X := X_1 \times \cdots \times X_n$ exists in $\Pfd$.
\item[(b)]
Write $\varphi_i$ as shorthand for $\varphi_{X_i}$, the automorphism of $X$ induced by the absolute ($p$-power) Frobenius on $X_i$. Then for any $i \in \{1,\dots,n\}$, the quotient 
\[
X/\Phi := X/\langle \varphi_1,\dots, \widehat{\varphi_i}, \dots,\varphi_n \rangle
\]
exists in $\Pfd$. (Note that it is canonically independent of the choice of $i$.)
\item[(c)]
If $X_1,\dots,X_n$ are qcqs, then $X/\Phi$ is qcqs (but $X$ need not be).
\end{enumerate}
\end{lemma}
\begin{proof}
In all three cases, it suffices to treat the case where $X_i = \Spa(A_i, A_i^+)$ is an affinoid perfectoid space
for $i=1,\dots,n$. 
(Here we use the convention that perfectoid rings are Tate, as per \cite{kedlaya-liu1, kedlaya-liu2},
and not merely analytic, as per \cite{kedlaya-aws}; this does not affect the definition of $\Pfd$.)
Let $\varpi_i \in A_i$ be a topologically nilpotent unit, let $A^+$ be the completion of $A_1^+ \otimes_{\FF_p} \cdots \otimes_{\FF_p} A_n^+$ for the
$(\varpi_1, \dots,\varpi_n)$-adic topology, and put $A := A^+[\varpi_1^{-1} \cdots \varpi_n^{-1}]$.
Then $(A,A^+)$ is a Huber pair, and taking
\[
X := \{x \in \Spa(A,A^+): \left|\varpi_1(x)\right|, \dots, \left|\varpi_n(x)\right| < 1 \}
\]
yields an absolute product of $X_1,\dots,X_n$; this proves (a). (Note that $X$ is not in general quasicompact.)

The action of $\Phi$ on $X$ (defined using $i=1$ for the sake of fixing notation)
is totally discontinuous: any $x \in X$ admits a neighborhood of the form
\[
\{x \in \Spa(A,A^+): \left| \varpi_1(x) \right|^{b_j} \leq \left|\varpi_j(x)\right| \leq \left| \varpi_1(x) \right|^{a_j} \quad (j=2,\dots,n)\}
\]
for some $a_j,b_j \in \QQ_{>0}$ with $a_j \leq b_j < pa_j$, and the $\Phi$-translates of such a neighborhood are
pairwise disjoint. Hence $X/\Phi$ exists in $\Pfd$, proving (b).
Moreover, it is clear that $X/\Phi$ is covered by finitely many such neighborhoods, and that the intersections
of these are again of the same form; it follows that $X/\Phi$ is quasicompact and separated (and in particular qcqs), proving (c).
\end{proof}

\begin{remark} \label{R:Pfd product redux}
Remark~\ref{R:diamond interpretation} amounts to the assertion that for $E$ of characteristic $p$,
\[
X_{L,E} \times_E F \cong (\Spa(L,L^\circ) \times \Spa(F,F^\circ))/\Phi.
\]
Consequently, Theorem~\ref{T:simple connectivity} implies that $\pi_1^{\prof}((\Spa(L,L^\circ) \times \Spa(F,F^\circ))/\Phi, \overline{x})$ is trivial (for any choice of the basepoint $\overline{x}$).
\end{remark}

As an intermediate step towards Theorem~\ref{T:drinfeld lemma perfectoid}, we establish the following analogue of 
Lemma~\ref{L:drinfeld schemes point}.

\begin{lemma} \label{L:extend by geometric point}
For $X \in \Pfd$, put $X_L := X \times \Spa(L,L^\circ)$ and let $\varphi_L: X_L \to X_L$ be the pullback of the Frobenius on $L$. Then the base extension functor
\[
\FEt(X) \to \FEt(X_L/\varphi_L)
\]
is an equivalence of categories. 
\end{lemma}
\begin{proof}
Since the functors $X \mapsto \FEt(X)$ and $X \mapsto \FEt(X_L/\varphi_L)$ on $\Pfd$ are both stacks for the \'etale topology, it suffices to check the claim when $X$ is a geometric point. 
This case is a consequence of Theorem~\ref{T:simple connectivity}, as described in
Remark~\ref{R:Pfd product redux}.
\end{proof}
\begin{cor} \label{C:extend by geometric point}
For $X \in \Pfd$ connected and qcqs, for any geometric point $\overline{x}$ of $X_L$,
the map $\pi_1^{\prof}(X_L/\varphi_L, \overline{x}) \to \pi_1^{\prof}(X, \overline{x})$
is a homeomorphism.
\end{cor}
\begin{proof}
This follows by combining Lemma~\ref{L:same pi1} with Lemma~\ref{L:extend by geometric point}.
\end{proof}

\begin{theorem} [Drinfeld's lemma for perfectoid spaces] \label{T:drinfeld lemma perfectoid}
Let $X_1,\dots,X_n \in \Pfd$ be connected and qcqs.
\begin{enumerate}
\item[(a)]
For $X := X_1 \times \cdots \times X_n$, the quotient $X/\Phi$ (which exists in $\Pfd$ by Lemma~\ref{L:Pfd products})
is connected.
\item[(b)]
For any geometric point $\overline{x}$ of $X$, the map
\[
\pi_1^{\prof}(X/\Phi, \overline{x}) \to \prod_{i=1}^n \pi_1^{\prof}(X_i, \overline{x})
\]
is an isomorphism of profinite groups.
\end{enumerate}
\end{theorem}
\begin{proof}
Put $X' := (X_1 \times \cdots \times X_{n-1})/\langle \varphi_2,\dots,\varphi_{n-1} \rangle$; then
\[
X/\Phi \cong (X' \times X_n)/\varphi_n,
\]
so by induction on $n$ we may deduce both claims from the case $n=2$.
(Note that this depends on the fact that by Lemma~\ref{L:Pfd products}, $X'$ exists in $\Pfd$ and is qcqs;
otherwise, we would have to formulate the induction hypothesis
in terms of a larger category than $\Pfd$.) We thus assume $n=2$ hereafter.

To check (a), we argue as in \cite[Corollary~4.1.22]{kedlaya-aws}
(which in turn is based on \cite[Proposition~16.3.6]{s-berkeley}).
Suppose by way of contradiction that $X$ admits a nontrivial $\varphi_2$-invariant disconnection $U_1 \sqcup U_2$.
For each geometric point $\overline{s}$ of $X_2$, 
$U_1 \times_{X_2} \overline{s}, U_2 \times_{X_2} \overline{s}$ form a $\varphi_2$-invariant
disconnection of $X_1 \times_{X_2} \overline{s}$; 
Lemma~\ref{L:same pi1} and Lemma~\ref{L:extend by geometric point} imply that this disconnection is the pullback of
a disconnection of $X_1$, and so one of the terms must be empty.
For $s = \Spa(K,K^+) \to X$ with $K$ a nonarchimedean field, we may write 
$U_1\times_{X_2} s$ as the inverse limit $\varprojlim_V X_1 \times_{X_2} V$ for $V$ running over quasicompact
open neighborhoods of $s$ in $X_2$; at the level of topological spaces, we have an inverse limit of spectral spaces and spectral morphisms, which can only be empty it if is empty at some term.
(For the constructible topologies, this is an inverse limit of compact Hausdorff spaces, which by Tikhonov's theorem cannot be empty if none of the terms is empty.)
It follows that the images of the maps $U_1 \to X_2$ and $U_2 \to X_2$ form a nontrivial disconnection of $X_2$,
yielding a contradiction.

To check (b), we continue as in \cite[Corollary~4.1.22, Corollary~4.1.23]{kedlaya-aws}.
In the diagram
\[
\xymatrix{
\pi_1^{\prof}(X \times_{X_2} \overline{s}, \overline{x}) 
\ar@{=}[dr] \ar[r] &  \pi_1^{\prof}(X/\Phi, \overline{x}) \ar[d]\ar[r]  & \pi_1^{\prof}(X_2, \overline{x}) \ar[r] & 1 \\
& \pi_1^{\prof}(X_1, \overline{x})
}
\]
the diagonal is an isomorphism by Lemma~\ref{L:same pi1} and Lemma~\ref{L:extend by geometric point} again.
For $X'_2 \in \FEt(X_2)$ connected, the previous paragraph implies that $(X_1 \times X'_2)/\Phi$ is connected;
as in \cite[Remark~4.1.10]{kedlaya-aws}, it follows that $\pi_1^{\prof}(X/\Phi, \overline{x}) \to \pi_1^{\prof}(X_2, \overline{x})$ is surjective. By the same token, $\pi_1^{\prof}(X/\Phi, \overline{x}) \to \pi_1^{\prof}(X_1, \overline{x})$ is surjective.

Let $G$ be a finite quotient of $\pi_1^{\prof}(X/\Phi, \overline{x})$ corresponding to $X' \in \FEt(X/\Phi)$.
Let $G \to H$ be the quotient corresponding to a Galois cover
$X_2' \to X_2$ as produced by  \cite[Proposition~16.3.3]{s-berkeley}
(the uniqueness property of that result implies the Galois property of the cover). 
Since $X' \to X_2'$ has geometrically connected fibers, the map
$\pi_1^{\prof}(X \times_{X_2} \overline{s}, \overline{x}) \to \ker(G \to H)$
must be surjective. This completes the proof of exactness.
\end{proof}

\section{Drinfeld's lemma in the analytic setting}
\label{sec:drinfeld diamonds}

To conclude, we reinterpret Theorem~\ref{T:simple connectivity} in Scholze's language of diamonds and v-sheaves \cite{s-berkeley, scholze-diamonds} and give the formulation of Drinfeld's lemma for v-sheaves described in
\cite[\S 4.3]{kedlaya-aws}, which see for more details. We conclude with a remark to the effect that Drinfeld's lemma for v-sheaves includes the corresponding statement for schemes (Remark~\ref{R:v-sheaves to schemes}).

\begin{defn}
Equip $\Pfd$ with the \emph{v-topology}, generated by open covers and arbitrary quasicompact surjective morphisms
\cite[Definition~17.1.1]{s-berkeley}. This topology is subcanonical, that is, representable presheaves
are sheaves \cite[Corollary~17.1.5]{s-berkeley}. A v-sheaf is \emph{small} if it admits a surjective morphism
from some perfectoid space.

For $\calF$ a small v-sheaf, we may define the \emph{underlying topological space} $\left| \calF \right|$ of
$\calF$ as follows: write $\calF$ as the quotient of a perfectoid space $X$ by an equivalence relation $R \subseteq X \times X$ using \cite[Proposition~17.2.2]{s-berkeley}, then set $\left| \calF \right| := \left| X \right| / \left| R \right|$. 
This does not depend on the choice of $X$ \cite[Proposition~12.7]{scholze-diamonds}.
We say $\calF$ is \emph{spatial} if it is qcqs and admits a neighborhood basis consisting of the underlying topological spaces of quasicompact open subobjects.

We use the term \emph{diamond} in the sense of \cite[Definition~8.3.1]{s-berkeley}, except
that we assume that all diamonds are locally spatial.
That is, diamonds are small v-sheaves (see \cite[Proposition~17.1.6]{s-berkeley})
which are locally the quotient of a perfectoid space by a pro-\'etale equivalence relation
(in the sense of \cite[\S 8.2]{s-berkeley}).
We write $X^\diamond$ for the diamond associated to the analytic adic space $X$ over $E$
\cite[Theorem~10.1.5]{s-berkeley}. 
For $A$ a Huber ring over $E$, we write $\Spd(A)$ as shorthand for $\Spa(A,A^\circ)^\diamond$.
\end{defn}

\begin{lemma} \label{L:Pfd products v-sheaves}
Let $X_1,\dots,X_n$ be locally spatial small v-sheaves (resp.\ diamonds)
\begin{enumerate}
\item[(a)]
The (absolute) product $X := X_1 \times \cdots \times X_n$ is a locally spatial small v-sheaf (resp.\ a diamond).
\item[(b)]
Write $\varphi_i$ as shorthand for $\varphi_{X_i}$, the automorphism of $X$ induced by the absolute ($p$-power) Frobenius on $X_i$. Then for any $i \in \{1,\dots,n\}$, the quotient 
\[
X/\Phi := X/\langle \varphi_1,\dots, \widehat{\varphi_i}, \dots,\varphi_n \rangle
\]
exists as a locally spatial small v-sheaf (resp.\ a diamond).
\item[(c)]
If $X_1,\dots,X_n$ are spatial, then $X/\Phi$ is spatial (but $X$ need not be).
\end{enumerate}
\end{lemma}
\begin{proof}
This is a formal consequence of Lemma~\ref{L:Pfd products}.
\end{proof}

\begin{remark}
Note that if $A$ is a perfectoid $E$-algebra, then $\Spd(A) = \Spd(A^\flat)$ as diamonds; the extra data of the choice of $A$ as an untilt of $A^\flat$ is equivalent to the specification of a structure morphism $\Spd(A^\flat) \to \Spd(E)$. 
Consequently, for any Huber ring $A$ over $E$, by Remark~\ref{R:Pfd product redux} we have an identification
\[
(Y_{L,E} \times_E \Spa(A,A^\circ))^\diamond \cong \Spd(L) \times \Spd(A).
\]
\end{remark}

\begin{remark} \label{R:symmetry}
Continuing with the previous remark, suppose that $F$ is a perfectoid field of characteristic $p$ containing $E$.
From Remark~\ref{R:diamond interpretation}, 
we have an identification
\[
\FEt((X_L \times_E F)^\diamond) \cong \FEt(\Spd(L) \times \Spd(F))/\varphi_L).
\]
On the other hand, as in Definition~\ref{D:Phi quotient}, for $X := \Spd(L) \times \Spd(F)$ 
we may identify the right side with
\[
\FEt(X/\Phi) \cong \FEt(X/\langle \varphi_L, \varphi_E \rangle) \times_{\FEt(X/\varphi_X)} \FEt(X),
\]
in which $L$ and $F$ play symmetric roles.
\end{remark}

Lemma~\ref{L:extend by geometric point} immediately promotes to the following statement.
\begin{lemma} \label{L:extend by geometric point v-sheaf}
For $X$ a small v-sheaf, put $X_L := X \times \Spa(L,L^\circ)$ and let $\varphi_L: X_L \to X_L$ be the pullback of the Frobenius on $L$. Then the base extension functor
\[
\FEt(X) \to \FEt(X_L/\varphi_L)
\]
is an equivalence of categories. 
\end{lemma}
\begin{proof}
This is a formal consequence of Lemma~\ref{L:extend by geometric point}.
\end{proof}

\begin{theorem}[Drinfeld's lemma for v-sheaves] \label{T:drinfeld lemma v-sheaves}
Let $X_1,\dots,X_n$ be connected spatial small v-sheaves. 
\begin{enumerate}
\item[(a)]
For $X := X_1 \times \cdots \times X_n$, the quotient $X/\Phi$ is connected.
\item[(b)]
For any geometric point $\overline{x}$ of $X$, the map
\[
\pi_1^{\prof}(X/\Phi, \overline{x}) \to \prod_{i=1}^n \pi_1^{\prof}(X_i, \overline{x})
\]
is an isomorphism of profinite groups.
\end{enumerate}
\end{theorem}
\begin{proof}
We cannot directly deduce this from Theorem~\ref{T:drinfeld lemma perfectoid}, because
a connected spatial small v-sheaf may not be covered by a connected perfectoid space.
However, given Lemma~\ref{L:Pfd products v-sheaves} and Lemma~\ref{L:extend by geometric point v-sheaf},
the proof of Theorem~\ref{T:drinfeld lemma perfectoid} adapts without incident.
\end{proof}

\begin{remark} \label{R:v-sheaves to schemes}
There is a fully faithful functor from arbitrary 
perfect schemes over $\FF_p$ to small v-sheaves
\cite[Proposition~18.3.1]{s-berkeley};
however, objects in the essential image of this functor need not be locally spatial,
so we cannot directly apply Theorem~\ref{T:drinfeld lemma v-sheaves} to them.

However, one may recover Theorem~\ref{T:drinfeld schemes} from Lemma~\ref{L:extend by geometric point v-sheaf} as follows. Suppose that $X$ is a perfect scheme over $\FF_p$ with associated v-sheaf $X^\diamond$;
we may still apply Lemma~\ref{L:extend by geometric point v-sheaf} to see that the target of the composition
\[
\FEt(X) \to \FEt(X_L/\varphi_L) \to \FEt((X^\diamond)_L/\varphi_L)
\]
is equivalent to $\FEt(X^\diamond) \cong \FEt(X)$. It is not hard to show that the second functor is fully
faithful, so both functors must be equivalences. This yields Lemma~\ref{L:drinfeld schemes point} for perfect schemes, from which one recovers Lemma~\ref{L:drinfeld schemes point} and Theorem~\ref{T:drinfeld schemes} at full strength
using Remark~\ref{R:drinfeld schemes perfect}.
\end{remark}


\begin{thebibliography}{99}

\bibitem{berkovich}
V. Berkovich, \textit{Spectral Theory and Analytic Geometry over Non-Archimedean Fields},
Math. Surveys and Monographs 33, Amer. Math. Soc., 1990.

\bibitem{bgr}
S. Bosch, U. G\"untzer, and R. Remmert,
\textit{Non-Archimedean Analysis},
Grundlehren der Math. Wiss. 261, Springer-Verlag, Berlin, 1984.

\bibitem{bourbaki-algebra}
N. Bourbaki, \textit{Alg\`ebre, Chapitres 4 \`a 7}, Springer-Verlag, Berlin, 2007.

\bibitem{ckz}
A.T. Carter, K.S. Kedlaya, and G. Z\'abr\'adi, Drinfeld's lemma for perfectoid spaces and overconvergence of multivariate $(\varphi, \Gamma)$-modules, arXiv:1808.03964v2 (2020).

\bibitem{christol-dwork}
G. Christol and B. Dwork, Modules diff\'erentielles sur les couronnes, \textit{Ann. Inst. Fourier} \textbf{44} (1994), 663--701.

\bibitem{cohen-temkin-trushkin}
A. Cohen, M. Temkin, and D. Trushkin,
Morphisms of Berkovich curves and the different function,
\textit{Adv. Math.} \textbf{303} (2016), 800--858.

\bibitem{ducros}
A. Ducros, \textit{La Structure des Courbes Analytiques}, draft (2014) available at
\url{https://webusers.imj-prg.fr/~antoine.ducros/livre.html}.

\bibitem{fargues-fontaine}
L. Fargues and J.-M. Fontaine, Courbes et fibr\'es vectoriels en th\'eorie de Hodge
$p$-adique, \textit{Ast\'erisque} \textbf{406} (2018).

\bibitem{grothendieck-vb}
A. Grothendieck, Sur la classification des fibr\'es holomorphes sur la sph\`ere de Riemann,
\textit{Amer. J. Math.} \textbf{79} (1957), 121--138.

\bibitem{hartl-pink}
U. Hartl and R. Pink, Vector bundles with a Frobenius structure on the punctured unit disc,
\textit{Compos. Math.} \textbf{140} (2004), 689--716.

\bibitem{huber-book}
R. Huber,
\textit{\'Etale Cohomology of Rigid Analytic Varieties and Adic Spaces},
Aspects of Mathematics, E30,
Friedr. Vieweg \& Sohn, Braunschweig, 1996.

\bibitem{kaplansky}
I. Kaplansky, Maximal fields with valuation, \textit{Duke Math. J.} \textbf{9} (1942), 303--321.

\bibitem{kaplansky2}
I. Kaplansky, Maximal fields with valuation, II, \textit{Duke Math. J.} \textbf{12} (1945), 243--248.

\bibitem{kedlaya-annals}
K.S. Kedlaya, A $p$-adic local monodromy theorem, \textit{Annals of Math.} \textbf{160} (2004), 93--184.

\bibitem{kedlaya-revisited}
K.S. Kedlaya, Slope filtrations revisited, \textit{Doc. Math.} \textbf{10} (2005), 447--525;
errata, \textit{ibid.} \textbf{12} (2007), 361--362; additional errata at
\url{http://kskedlaya.org/papers/}.

\bibitem{kedlaya-swan1}
K.S. Kedlaya, 
Swan conductors for $p$-adic differential modules, I: A local construction, 
\textit{Algebra and Num. Theory} \textbf{1} (2007), 269--300.

\bibitem{kedlaya-swan2}
K.S. Kedlaya,
Swan conductors for $p$-adic differential modules, II: Global variation, 
\textit{J. Inst. Math. Jussieu} \textbf{10} (2011), 191--224.

\bibitem{kedlaya-relative}
K.S. Kedlaya, Slope filtrations for relative Frobenius,
\textit{Ast\'erisque} \textbf{319} (2008), 259--301.

\bibitem{kedlaya-book}
K.S. Kedlaya, \textit{$p$-adic Differential Equations},
Cambridge Univ. Press, Cambridge, 2010.

\bibitem{kedlaya-nonarch}
K.S. Kedlaya, Nonarchimedean geometry of Witt vectors, \textit{Nagoya Math. J.} \textbf{209} (2013), 111--165.

\bibitem{kedlaya-noetherian}
K.S. Kedlaya,
Noetherian properties of Fargues-Fontaine curves,
\textit{Int. Math. Res. Notices} (2015), article ID rnv227.

\bibitem{kedlaya-conv}
K.S. Kedlaya, Local and global structure of connections on nonarchimedean curves,
\textit{Compos. Math.} \textbf{151} (2015), 1096--1156.

\bibitem{kedlaya-simons}
K.S. Kedlaya, Convergence polygons for connections on nonarchimedean curves, in 
\textit{Nonarchimedean and Tropical Geometry}, Simons Symposia, Springer Nature, 2016, 51--98. 

\bibitem{kedlaya-aws}
K.S. Kedlaya, Sheaves, stacks, and shtukas, in
\textit{Perfectoid Spaces: Lectures from the 2017 Arizona Winter School},
Math. Surveys and Monographs 242, Amer. Math. Soc., 2019, 58--205.

\bibitem{kedlaya-liu1}
K.S. Kedlaya and R. Liu, Relative $p$-adic Hodge theory: Foundations,
\textit{Ast\'erisque} \textbf{371} (2015).

\bibitem{kedlaya-liu2}
K.S. Kedlaya and R. Liu, Relative $p$-adic Hodge theory, II: Imperfect period rings,
arXiv:1602.06899v3 (2019).

\bibitem{krasner1}
M. Krasner, Prolongement analytique uniforme et multiforme dans les corps valu\'es complets,
in \textit{Les Tendances G\'eometrique et Alg\'ebrique et Th\'eorie des Nombres},
Centre National de la Recherche Scientifique, Paris, 1966, 97--141.

\bibitem{krasner2}
M. Krasner, Rapport sur le prolongement analytique dans les corps valu\'es complets par la methode des elements analytique quasi-connexes, \textit{M\'em. Soc. Math. France} \textbf{39--40} (1974), 131--254.

\bibitem{narasimhan-seshadri}
M.S. Narasimhan and C.S. Seshadri, Stable and unitary vector bundles on a compact Riemann surface,
\textit{Ann. of Math.} \textbf{82} (1965), 540--567.

\bibitem{scholze1}
P. Scholze, Perfectoid spaces,
\textit{Publ. Math. IH\'ES} \textbf{116} (2012), 245--313.

\bibitem{s-berkeley}
P. Scholze and J. Weinstein, \textit{Berkeley Lectures on $p$-adic Geometry}, 
Annals of Math. Studies 207, Princeton Univ. Press, 2020.

\bibitem{scholze-diamonds}
P. Scholze, \'Etale cohomology of diamonds,
preprint (2017) available at \url{http://www.math.uni-bonn.de/people/scholze/}.

\bibitem{stacks-project}
The Stacks Project Authors, \textit{Stacks Project}, \url{http://stacks.math.columbia.edu}
(retrieved Nov 2017).

\bibitem{temkin-curves}
M. Temkin, Wild coverings of Berkovich curves, to appear in \textit{Publ. Math. Besan\c{c}on}.

\bibitem{weinstein-gqp}
J. Weinstein, $\text{Gal}(\overline{\mathbf{Q}}_p/\mathbf{Q}_p)$ as a geometric fundamental group,
\textit{Int. Math. Res. Notices} (2017), 2964--2997.

\bibitem{weinstein-aws}
J. Weinstein, Adic spaces, 
\textit{Perfectoid Spaces: Lectures from the 2017 Arizona Winter School},
Math. Surveys and Monographs 242, Amer. Math. Soc., 2019, 14--57.

\bibitem{xiao-swan1}
L. Xiao, On ramification filtrations and $p$-adic differential equations, I: equal characteristic case, 
\textit{Algebra and Num. Theory} \textbf{4} (2010), 969--1027.

\bibitem{xiao-swan2}
L. Xiao, On ramification filtrations and $p$-adic differential equations, II: mixed characteristic case, 
\textit{Compos. Math.} \textbf{148} (2012), 415--463.

\end{thebibliography}
\end{document}